\documentclass{article}

\PassOptionsToPackage{numbers,compress}{natbib}

\usepackage[main,final]{neurips_2026}

\usepackage[utf8]{inputenc} 
\usepackage[T1]{fontenc}    
\usepackage[colorlinks=true,allcolors=black]{hyperref}
\usepackage{url}            
\usepackage{booktabs}       
\usepackage{amsfonts}       
\usepackage{nicefrac}       
\usepackage{microtype}      
\usepackage{xcolor}         

\title{Bentkus-type asymptotic e-values}

\author{
 Diego Martinez-Taboada$^{1}$, Ben Chugg$^{12}$, and Aaditya Ramdas$^{12}$ \\ \\ 
  $^1$Department of Statistics \& Data Science\\
$^2$Machine Learning Department \\
  Carnegie Mellon University\\
  \texttt{\{diegomar,bchugg,aramdas\}@andrew.cmu.edu} }

\usepackage{graphicx}
\usepackage{amsmath,amssymb,amsthm}

\usepackage{xcolor}
\usepackage{enumitem}
\usepackage{setspace}

\allowdisplaybreaks

\usepackage{xspace}

\newcommand{\Pb}{\mathbb{P}}
\newcommand{\Nb}{\mathbb{N}}
\newcommand{\R}{\mathbb{R}}

\newcommand{\E}{\mathbb{E}}
\newcommand{\V}{\mathbb{V}}

\newcommand{\F}{\mathcal{F}}

\newcommand{\lp}{\left(}
\newcommand{\rp}{\right)}
\newcommand{\lc}{\left\{}
\newcommand{\rc}{\right\}}

\newtheorem{theorem}{Theorem}[section]

\newtheorem{lemma}[theorem]{Lemma}
\newtheorem{proposition}[theorem]{Proposition}
\newtheorem{corollary}[theorem]{Corollary}

\newtheorem{remark}[theorem]{Remark}

\newtheorem{fact}[theorem]{Fact}

\begin{document}

\maketitle

\begin{abstract}

Asymptotic e-values are emerging as a powerful alternative to asymptotic p-values, particularly in post-hoc inference and multiple testing, where significance levels may be data-dependent. Existing asymptotic e-values, however, suffer from the ``missing factor,'' a scaling inefficiency resulting in overly conservative inference. Drawing on the framework of near-optimal concentration inequalities developed by Bentkus in the 2000s, we introduce Bentkus-type asymptotic e-values and prove that they successfully eliminate the missing factor. We also demonstrate both theoretically and empirically that Bentkus-type e-values can deliver sharper inference than exponential asymptotic e-values, leading to tighter post-hoc confidence intervals and higher rejection rates in multiple testing procedures.

\end{abstract}

\section{Introduction} \label{section:introduction}

E-values have recently emerged as a versatile alternative to p-values for statistical inference~\citep{ramdas2025hypothesis}. 
They offer several advantages: they often remain valid under optional stopping~\citep{grunwald2024safe}, combine easily under arbitrary dependence, and exist for irregular problems where no other inferential method is known~\citep{wasserman2020universal}, among others. 
Beyond being useful, they have also proven \emph{necessary} in various problems,  such as multiple testing~\citep{wang2022false,fischer2024admissible,xu2025bringing}, statistical contract theory~\citep{bates2022principal}, and post-hoc inference~\citep{grunwald2024beyond}. 

Formally, an e-value is a nonnegative test statistic whose expected value is at most one under the null hypothesis. Ideally, analysts want e-values that are large under the alternative---that is, e-values with high power. Such  e-values have been constructed under a variety of distributional assumptions, including for bounded random variables~\citep{grunwald2024safe, waudby2024estimating, martinez2025sharp, martinez2024empirical, martinez2026intrinsic}, random variables whose variance is upper bounded by a known constant~\citep{catoni2012challenging, howard2020time, howard2021time, wang2023catoni,whitehouse2026nonasymptotic, martinez2025vector}, and parametric settings~\citep{hao2024values,grunwald2024optimal,de2025rao}, among others. 

Despite this progress, no powerful e-value is known for the prominent case of random variables with no a priori known upper bound on the variance. More precisely, if $X_1, \ldots, X_n$ are i.i.d. with mean $\E X_i = \mu$ and non-zero finite variance $\sigma^2 = \V X_1 \in (0, \infty)$, there is no known nontrivial e-value (and it is likely that none exist) for null hypothesis
\begin{align} \label{eq:main_null_hypothesis}
    H_0: \mu \leq \theta \quad \text{for some fixed } \theta \in \R.
\end{align}
For this reason, \textit{asymptotic} e-values have recently been proposed~\citep{ignatiadis2024asymptotic}. As usual in the world of asymptotics, such e-values are designed to give guarantees holding in the limit (i.e., they only hold \textit{approximately} for big $n$). Unfortunately, existing asymptotic e-values for~\eqref{eq:main_null_hypothesis} are provably suboptimal. In particular, they fall prey to an issue that has been called ``the missing factor'' in other contexts~\citep{talagrand1995missing,kuchibhotla2024missing}. This work designs asymptotic e-values which do not suffer this drawback.

\paragraph{The missing factor.}
All existing asymptotic e-values are based on the exponential function. To make the problem of the missing factor explicit, we will discuss the asymptotic e-value proposed by \citet{ignatiadis2024asymptotic} (who first also defined asymptotic e-values). Let $X_1,\dots,X_n$ be i.i.d.\ and set 
\begin{align*}
    S_n(\theta) := \sum_{i = 1}^n (X_i - \theta), \quad V_n(\theta) := \bigg(\sum_{i = 1}^n (X_i-\theta)^2\bigg)^{1/2} , \text{~ and ~} Z_n(\theta) := \frac{S_n(\theta)}{V_n(\theta)}. 
\end{align*}
Consider the random variable 
\begin{align} \label{eq:exponential_e_value}
    E_{n}^{\infty}(\theta; \lambda) := \exp \lp \lambda Z_n(\theta) - {\lambda^2}/{2}\rp. 
\end{align}
\citet{ignatiadis2024asymptotic} prove that $(E_n^\infty)$
is an asymptotic e-value for $H_0$
under the assumption of finite, nonzero variance. This has since been relaxed to $(X_n)$ lying in the domain of attraction of a normal distribution~\citep{chugg2026post}, a strictly weaker condition~\citep{gine1997student}.

Inference based on $E_n^\infty$ proceeds by thresholding: One rejects $H_0$ when $E_n^\infty \geq 1/\delta$ for some $\delta\in(0,1)$. After optimizing $\lambda$, this is equivalent to $Z_n(\theta)$ exceeding the threshold $\sqrt{2\log(1/\delta)}$. 
By contrast, CLT-based methods---such as the Wald confidence interval---depend on the $\delta$-th quantile of the standard normal $z_\delta = G^{-1}(\delta)$,  
where $G$ is the survival function. 
The discrepancy between $z_\delta$ and $\sqrt{2\log(1/\delta)}$ is called \emph{the missing factor}---showcased in more detail in Section~\ref{sec:background}---and means that existing asymptotic e-values result in looser inference than their p-value counterparts.  


The term \emph{missing factor} comes from the \emph{nonasymptotic} concentration inequality literature~\citep{talagrand1995missing,fan2012missing,kuchibhotla2024missing}. 
In particular, Cramér-Chernoff bounds are exponential function-based inequalities, and are known to also suffer from a $\sqrt{2\log(1/\delta)}$ dependence. 
This suboptimality resulted in a flurry of activity, culminating in more sophisticated analyses which close the gap~\citep{bentkus2002remark,bentkus2004hoeffding} and led to tighter inequalities whose $\delta$-dependent factor is at most $G^{-1}( \delta / c)$, with $c$ being a small and known constant.

\paragraph{Our contributions.}
The fact that exponential-based \emph{asymptotic} e-values suffer from the missing factor suggests a question: Do there exist asymptotic e-values for the null hypothesis~\eqref{eq:main_null_hypothesis} with a smaller $\delta$-dependent factor?\footnote{The straightforward two-sided extension is exhibited in Appendix~\ref{section:two_sided}.} In particular, can we construct asymptotic e-values whose 
$\delta$-dependent term scales as $G^{-1}(\delta / c)$ and which lead to sharper inference in practice?
Our contribution is to answer these questions in the affirmative. 
We do so by establishing a novel family of asymptotic e-values that we refer to as \textit{Bentkus-type asymptotic e-values}, in honor of Vidmantus Bentkus, who pioneered near-optimal concentration inequalities~\citep{bentkus2002remark,bentkus2004hoeffding,bentkus2006domination}, and whose theoretical tools are the foundation of this work.

Like Bentkus' bounds in the nonasymptotic setting, our Bentkus-type asymptotic e-values are near-optimal in the sense that their $\delta$-dependent factor is at most $G^{-1}(\delta/c)$ for a known constant $c>0$. 
For fixed values of $\delta$ in realistic ranges ($\leq 0.1$) we prove that our Bentkus-type e-values lead to tighter inference than exponential-based asymptotic e-values. 

We then apply our results to two areas where asymptotic e-values are commonly used: post-hoc inference and multiple testing. Multiple-testing is, of course, a cornerstone of modern statistical analysis, while post-hoc inference is a new area of study which seeks to construct confidence sets that allow the significance level to be data-dependent. In both applications, we show that Bentkus-type asymptotic e-values can improve existing results, yielding sharper post-hoc confidence sets and more  rejections in multiple-testing procedures.

\paragraph{Outline.} 
Section~\ref{sec:motivation} provides background on asymptotic e-values, post-hoc inference, and multiple testing. Section~\ref{sec:background} reviews exponential-based asymptotic e-values and makes the missing factor precise, and also recalls the near-optimal concentration inequalities of Bentkus on which our construction rests. Section~\ref{sec:bentkus-type} introduces Bentkus-type asymptotic e-values, establishes their validity and near-optimality (Theorems~\ref{theorem:main_theorem}--\ref{proposition:near_optimality}), and analyzes their behavior under data-dependent choices of $\delta$. Section~\ref{section:applications} demonstrates their practical gains in post-hoc inference and multiple testing. Proofs are collected in Appendix~\ref{section:proofs}.

\section{Motivation and Related Work}
\label{sec:motivation}
 
An asymptotic e-value for a set of distributions $\mathcal{P}$ is a sequence of nonnegative random variables $(E_n)_{n \geq 1}$ which satisfies 
\begin{align}
\label{eq:asymp-evalues}
    \limsup_{n \to \infty} \E_P[E_n] \leq 1 \text{ for all } P \in \mathcal{P}.
\end{align}
We typically suppress the dependence on $\mathcal P$ when it is clear from context. Asymptotic e-values were introduced by \citet{ignatiadis2024asymptotic},\footnote{\citet{ignatiadis2024asymptotic} gave both strong and weak variants. The definition we give above is equivalent to the former.}
and should be thought of as in the same family of statistical objects as asymptotic p-values---both enable inference and can be interpreted as measures of evidence against the null hypothesis. However, as is the case for \emph{nonasymptotic} e-values and p-values, asymptotic e-values possess various properties that asymptotic p-values do not.

For instance, convex combinations of asymptotic e-values remain asymptotic e-values under arbitrary dependence, making them particularly powerful in multiple-testing procedures~\citep{xu2024post,wang2022false}.  Furthermore, asymptotic e-values have the property that they enable post-hoc analysis---inference and testing under data-dependent significance levels. 


\paragraph{Post-hoc inference.}
Post-hoc confidence intervals replace fixed-level coverage guarantees with a stronger form of risk control.  Let $\theta(P)$ denote the parameter associated
with a distribution $P\in\mathcal P$. We say that a sequence of sets $(C_{n,\delta})_{n\geq 1}$ is an
asymptotic post-hoc confidence set for $\theta(P)$ over a class $\mathcal P$ if, for all $P\in\mathcal P$,
\begin{equation}
\label{eq:risk-control}
\limsup_{n\to\infty}
\E_P\left[
    \sup_{\delta>0}
    \frac{\mathbf 1\{\theta(P)\notin C_{n,\delta}\}}{\delta}
\right]
\leq 1 .
\end{equation}
This should be contrasted with the usual fixed-level requirement
\[
    \limsup_{n\to\infty}
    \frac{P\{\theta(P)\notin C_{n,\delta}\}}{\delta}
    \leq 1
    \qquad \text{for each fixed } \delta>0 .
\]
The difference is the placement of the supremum.  Fixed-level coverage controls the error probability after $\delta$ has been chosen in advance, whereas~\eqref{eq:risk-control} controls the worst error over all levels simultaneously. Consequently, the level $\delta$ may be chosen after seeing the data.


There has been considerable recent attention on post-hoc confidence intervals and hypothesis tests~\citep{grunwald2024beyond,koning2025post,chugg2026admissibility,koobs2026equivalence,gauthier2025values}, primarily owing to the realization that e-values are both necessary and sufficient for their construction. While most of this work is on the nonasymptotic setting, \citet{chugg2026post} recently began the study of asymptotic post-hoc inference, introducing the guarantee in~\eqref{eq:risk-control}. 
Suppose that $\mathcal P = \{P_\theta:\theta\in\Theta\}$ and that we have a family of sequences $\{(E_n(\theta))_{n\geq 1}:\theta\in\Theta\}$ where $(E_n(\theta))$ is an asymptotic e-value for all distributions with parameter $\theta$. \citet{chugg2026post} show that
\begin{equation}
    C_{n,\delta} = \{ \theta: E_n(\theta) < 1/\delta\},
\end{equation}
forms an asymptotic post-hoc confidence set for $\mathcal P$. Unlike traditional (fixed-level) inference, however, $E_n(\theta)$ cannot depend on $\delta$, which is where post-hoc inference departs from traditional inference. \citet{chugg2026post} explore two techniques for handling this: ex-ante anchoring and the method of mixtures. We discuss these further in Section~\ref{sec:background}, and Section~\ref{sec:bentkus-type}, will demonstrate how our Bentkus-type asymptotic e-values offer improvements to both alternatives. 

\paragraph{Multiple-testing.} The goal of multiple testing is to simultaneously test $K$ hypotheses while controlling some global error metric, such as the false discovery rate (FDR). The Benjamini-Hochberg (BH) procedure~\citep{benjamini1995controlling}, which is based on p-values, is the gold standard for FDR control under independence or positive dependence. However, traditional p-value-based methods often struggle with arbitrary dependence structures, requiring conservative corrections  \citep{benjamini2001control} which can significantly reduce power. 

E-values have recently emerged as a powerful alternative for multiple testing, primarily through the e-BH procedure \citep{wang2022false, ramdas2025hypothesis}. Unlike p-values, e-values can be easily combined under arbitrary dependence, and the e-BH procedure provides FDR control for any dependence structure without the need for additional corrections. \citet{ignatiadis2024asymptotic} extended these concepts to asymptotic settings, where the guarantees hold in the limit (i.e., approximately). In particular, the e-BH procedure at level $\delta^*$ rejects hypotheses with the $k^*$ largest e-values, where
\begin{align*}
    k^* = \max \lc k \in \{ 1 , \ldots, K \}: \frac{k E_{[k]}}{K} \geq \frac{1}{\delta^*} \rc,
\end{align*}
where $E_{[k]}$ is the $k$-th order statistic of $E_1, \ldots, E_K$, from the largest to the smallest, and with the convention $\max(\emptyset) = 0$. 
Our asymptotic FDR guarantee takes $n\to\infty$ with $K$ fixed. For $K=K_n\to\infty$, a sufficient additional condition is $K_n^{-1}\sum_{k\in\mathcal H_{0,n}}\E[E_{n,k}]\leq1+o(1)$, since e-BH satisfies $\mathrm{FDR}\leq(\delta^*/K_n)\sum_{k\in\mathcal H_{0,n}}\E[E_{n,k}]$.  

\section{Background on Exponential E-values and Near-Optimal Inequalities}
\label{sec:background}

\textbf{Exponential-based asymptotic e-values.}
As discussed earlier, an asymptotic e-value is used for inference 
by thresholding it at level $1/\delta$. 
In the case of~\eqref{eq:exponential_e_value}, we obtain
\begin{align*}
    E_{n}^{\infty}(\theta; \lambda) \geq \frac{1}{\delta} \iff Z_n(\theta) \geq \frac{\log(1/\delta) + \lambda^2 / 2}{\lambda} =: U_{\delta, \infty}(\lambda).
\end{align*}
Note that smaller values of $U_{\delta, \infty}(\lambda)$ lead to tighter inference. 
If $\delta$ is fixed, the minimum of $U_{\delta, \infty}(\lambda)$ is 
\begin{align*}
    \inf_{\lambda \geq 0} U_{\delta, \infty}(\lambda) = \inf_{\lambda \geq 0} \frac{\log(1/\delta) + \lambda^2 / 2}{\lambda} = \sqrt{2\log(1/\delta)} =: \zeta_\delta,
\end{align*}
in which case we obtain the one-sided confidence interval $\{\theta: S_n(\theta)/V_n(\theta) < \zeta_\delta\}$. 
By contrast, the classical central limit theorem leads to (one-sided) confidence sets of the form $\{\theta: Z_n(\theta) < z_\delta\}$ where $z_\delta$, as above, is the $\delta$-th quantile of a standard normal. For $0<\delta\leq0.1$,
\begin{align*}
    \exp\lp -\zeta_\delta^2/2\rp = \delta = G(z_\delta)  \leq \frac{\exp\lp -z_\delta^2/2\rp}{\sqrt{2\pi(z_\delta^2 + 1)}},
\end{align*}
which implies $z_\delta < \zeta_\delta$ and furthermore there existsx a scaling difference of $\sqrt{2\pi(u+1)}$ (where $u = z_\delta^2$ takes the value of the corresponding Gaussian quantile). 
This difference is the missing factor, discussed in the introduction. Table~\ref{table:missing_factor} illustrates the missing factor for various values of $\alpha$ and $\delta$.  

\begin{table}[h!]
\centering
\caption{Illustration of the missing factor. The CLT threshold $z_\delta = G^{-1}(\delta)$ is the ideal baseline; the exponential e-value threshold $\zeta_\delta = \sqrt{2\log(1/\delta)}$ illustrates the missing factor; and $\inf_\lambda U_{\delta,\alpha}$ gives the optimal Bentkus-type threshold for $\alpha \in \{0,1,2\}$ (Theorem~\ref{proposition:near_optimality}). Note that $\inf_\lambda U_{\delta,0} = z_\delta$, i.e., $\alpha=0$ recovers the missing factor exactly. 
}
\label{table:missing_factor}
\begin{tabular}{lrrrrrr}
\toprule
$\delta$ & $z_\delta$ & $\inf_\lambda U_{\delta,0}$ & $\inf_\lambda U_{\delta,1}$ & $\inf_\lambda U_{\delta,2}$ & $\zeta_\delta$ \\
\midrule
$0.1$    & 1.282 & 1.282 & 1.755 & 1.867 & 2.146 \\
$0.05$   & 1.645 & 1.645 & 2.063 & 2.166 & 2.448 \\
$0.01$   & 2.326 & 2.326 & 2.665 & 2.754 & 3.035 \\
$0.001$  & 3.090 & 3.090 & 3.367 & 3.443 & 3.717 \\
$0.0001$ & 3.719 & 3.719 & 3.958 & 4.026 & 4.292 \\
\bottomrule
\end{tabular}

\end{table}

\textbf{Near-optimal concentration inequalities.}
Consider the family $(h_\alpha)_{\alpha \in [0, \infty]}$ of functions $h_\alpha: \R \to \R$ given by the formula
\begin{align*}
    h_\alpha(u):= 
    \begin{cases} 
        \mathbf{1} \lc u \geq 0 \rc & \text{if } \alpha = 0, \\
        (1 + u/\alpha)_+^\alpha   & \text{if } 0 < \alpha < \infty, \\
        e^{u} & \text{if } \alpha = \infty,
    \end{cases}
\end{align*}
for all $u \in \R$, 
where $\mathbf{1} \lc \cdot\rc$ denotes the indicator function, $x_+ = \max\{0,x\}$, and $x_+^\alpha = (x_+)^\alpha$. As noted in \citet{pinelis2014optimal}, the function $h_\alpha$ is nonnegative and nondecreasing for each $\alpha \in [0, \infty]$, and it is also continuous for each $\alpha \in (0, \infty]$. Moreover, for each $u \in \R$, $h_\alpha(u)$ is non decreasing and continuous in $\alpha \in [0, \infty]$. We call such functions \emph{$\alpha$-powered positive functions}. 

It is easy to see that $h_\alpha(u) \leq e^u$ for each $\alpha\in[0,\infty]$, with a strict inequality for $\alpha<\infty$. 
Thus, $h_\alpha$ provides a family of positive functions that interpolate between the indicator and the exponential, which leads to tighter inference than the Cramér-Chernoff method \citep{bentkus2002remark, bentkus2003inequality, bentkus2004hoeffding, pinelis2006binomial,pinelis2006normal, pinelis2014bennett, pinelis2014optimal}.   
To elaborate, suppose we want to bound $P(X\geq \eta)$ for some random variable $X$. For any $\alpha\in[0,\infty)$, let $\gamma = \alpha / (\eta-\lambda)$ where $\lambda <\eta$ and write $P(X\geq \eta) = \E\mathbf{1}\{\gamma(X-\eta)\geq 0\} \leq  \E h_\alpha(\gamma(X-\eta)) = \E f_{\lambda,\eta,\alpha}(X),$
where 
\begin{align} \label{eq:f_definition}
    f_{\lambda, \eta, \alpha}(r) = \frac{(r-\lambda)_+^\alpha}{(\eta-\lambda)^\alpha}, \quad r \in \R.
\end{align}
While this bound is tighter than that of the Cramér-Chernoff method (since $h_\alpha(\gamma (X-\eta))\leq \exp\{\gamma(X-\eta)\}$), controlling the expectation of $f_{\lambda,\eta,\alpha}(X)$ is more challenging. 

The method used by \citet{bentkus2004hoeffding} and \citet{pinelis2006binomial} was to show that the distributions which maximize $\E f_{\lambda,\eta,\alpha}(X)$ are binomial distributions, and then upper bound this expectation. 
Our approach differs from this strategy, as our analysis relies on Gaussian distributions instead. Importantly, an upper bound for any arbitrary pair $(\lambda, \eta)$ is not needed; it suffices to consider specific (and optimal) pairs $(\lambda, \eta_\lambda)$.  
We will thus borrow ideas from a related contribution of \citet{bentkus2006domination}, who studied the \textit{transformed survival function }
\begin{align} \label{eq:transformed_survival_function}
     G_\alpha (\eta) := \inf_{\lambda < \eta} \E f_{\lambda, \eta, \alpha} (Z) = \E f_{\lambda_\eta, \eta, \alpha} (Z)
\end{align}
for Gaussian $Z$. One of the core ideas of our contribution is to exploit Bentkus' upper bounds for $\E f_{\lambda, \eta_\lambda, \alpha}(Z)$ in order to establish the near-optimality of the proposed asymptotic e-values.


\section{Bentkus-type asymptotic e-values}
\label{sec:bentkus-type}



We are now ready to present the main results of this paper:  asymptotic e-values based on the $\alpha$-powered positive functions in Section~\ref{sec:background}. Such \textit{Bentkus-type} asymptotic e-values (Theorem~\ref{theorem:main_theorem}) retain the near-optimality guarantees of their related concentration inequalities for a fixed level (Theorem~\ref{proposition:near_optimality}). This results in a novel family of asymptotic e-values that is less conservative than exponential-based asymptotic e-values, leading to sharper inference. 
We defer all proofs to Appendix~\ref{section:proofs}, and a higher level discussion on the connection between evidence, testing, exponential functions, and $\alpha$-powered functions to Appendix~\ref{section:evidence_vs_testing}.


\subsection{Deriving Bentkus-type asymptotic e-values} \label{section:near_optimal_asymptotic_e_values}

Fix $\lambda \in \R$ and $\alpha \in [0, \infty)$, and define
\begin{align} \label{eq:alpha_powered_e_value}
    E_{n}^{\alpha}(\theta; \lambda) := \frac{\lp {Z_n(\theta)} -\lambda \rp_+^\alpha}{I_\alpha(\lambda)}, \quad I_\alpha(\lambda):= \E \lp {Z} -\lambda \rp_+^\alpha,
\end{align}
where $Z$ is a standard Gaussian. For the optimizer $\lambda_\eta$ in~\eqref{eq:transformed_survival_function}, with $f$ defined in~\eqref{eq:f_definition}, we can write $E_{n}^{\alpha}(\theta; \lambda_\eta) = {f_{\lambda_\eta, \eta, \alpha}(Z_n(\theta))}/{\E f_{\lambda_\eta, \eta, \alpha}(Z)}$, elucidating the connection between~\eqref{eq:transformed_survival_function} and~\eqref{eq:alpha_powered_e_value}.
However, $E_{n}^{\alpha}$ does not depend on $\eta$, as the denominators cancel out, reducing to the truncated (and uncentered) $\alpha$-moments of a Gaussian distribution.

We start by proving in Theorem~\ref{theorem:main_theorem} that $E_n^\alpha$ is indeed an asymptotic e-value for any i.i.d.\ data with mean $\theta$ which lies in the domain of attraction of a Gaussian.\footnote{The analysis is tailored to Gaussian limiting laws, but the methodology is more general, as we discuss in Appendix~\ref{section:other_limits}.} 
 Recall that $(X_n)$ are said to be in the domain of attraction of a Gaussian if there exist sequences $(a_n)$ and $(b_n)$ such that
 \begin{align} \label{eq:domain_attraction_gaussian}
     \frac{\sum_{i \leq n} X_i - b_n}{a_n} \stackrel{d}{\to} N(0,1).
 \end{align}
 If the distribution has finite variance $\sigma^2$ and mean $\theta$, then~\eqref{eq:domain_attraction_gaussian} holds with $a_n = \sigma\sqrt{n}$ and $b_n = n \theta$. Thus, lying in the domain of attraction of a Gaussian is strictly more general than being i.i.d.\ with finite variance. 
 For i.i.d. self-normalized sums, this domain-of-attraction condition characterizes asymptotic normality~\citep{gine1997student}. 
 A key component of the proof of Theorem~\ref{theorem:main_theorem}, which is deferred to Appendix~\ref{proof:main_theorem}, is to establish the uniform integrability of the sequence of $\alpha$-powered positive functions.
 
\begin{theorem} \label{theorem:main_theorem}
    Let $X_1,\dots,X_n$ be i.i.d.\ with mean $\theta$ and lying in the domain of attraction of a Gaussian. For any $\lambda \in \R$ and $\alpha \in [0, \infty)$, 
    the sequence $(E_{n}^{\alpha}(\theta; \lambda))_{n\geq 1}$ is an asymptotic e-value. 
\end{theorem}

In order to use asymptotic e-values $E_{n}^{\alpha}(\theta; \lambda)$, one needs to be able to evaluate $I_\alpha(\lambda) = \E \lp {Z} -\lambda \rp_+^\alpha$. In Lemma~\ref{lemma:fracional_moments}, we derive $I_\alpha$ in terms of parabolic cylinder functions, whose evaluation can be obtained via well-established numerical approximations. When $\alpha$ is an integer, the truncated moments $I_\alpha$ can also be obtained recursively as a function of $\phi$ and $G$ solely, leading to simpler analytical expressions. We formalize the result in the next lemma, whose proof is deferred to Appendix~\ref{proof:i_recursion}.

\begin{lemma} \label{lemma:i_recursion}
The base cases for the truncated moments are $I_0(\lambda) = G(\lambda)$ and $I_1(\lambda) = \phi(\lambda) - \lambda G(\lambda)$. Furthermore, for any natural $\alpha \ge 2$,
\begin{align} \label{eq:recurrence}
I_\alpha(\lambda) = (\alpha-1) I_{\alpha-2}(\lambda) - \lambda I_{\alpha-1}(\lambda).
\end{align}
\end{lemma}

As discussed in Section~\ref{section:introduction}, the asymptotic e-values $E_{n}^{\alpha}(\theta; \lambda)$ will be used for inference by thresholding them by $1/\delta$ (that is, $E_{n}^{\alpha}(\theta; \lambda) \geq {1}/{\delta}$). For $\alpha \in (0, \infty)$, this holds if and only if
\begin{align*}
     Z_n(\theta) \geq U_{\delta, \alpha}(\lambda) := \lambda + \lp { I_\alpha(\lambda)}/{\delta} \rp^{\frac{1}{\alpha}}, \quad \lambda \in \R.
\end{align*}
For $\alpha = 0$, we obtain
\begin{align}
    U_{\delta, 0}(\lambda) =
    \begin{cases} 
        \lambda & \text{if } G(\lambda) \le \delta,\\
        \infty   & \text{if }  G(\lambda) > \delta.
    \end{cases}
\end{align}
The optimal thresholds $\inf_\lambda U_{\delta, \alpha}(\lambda)$
inherit the near-optimality of Bentkus-type bounds for fixed $\delta$. 
The result is formalized in the following theorem, whose proof can be found in Appendix~\ref{proof:near_optimality}.

\begin{theorem} [Near-optimality of Bentkus-type asymptotic e-values for fixed $\delta$]
\label{proposition:near_optimality} 
For any $\alpha \in [0, \infty)$ and any $\delta \in (0,1)$, it holds that
\begin{align} \label{eq:u_sandwich}
    G^{-1}(\delta) \leq \inf_\lambda U_{\delta, \alpha}(\lambda) \leq G^{-1}(\delta / c_\alpha),
\end{align}
where $ c_\alpha = e^{\alpha} \alpha^{-\alpha} \Gamma(\alpha + 1)$ and $\Gamma(s) := \int_0^\infty t^{s-1}e^{-t} dt$.
\end{theorem}


In particular, $G^{-1}(\delta/c_\alpha) < \zeta_\delta$ for $\delta < \tau(\alpha)$, where 
\begin{align*}
\tau(\alpha) := \exp\left( -  \frac{1}{2}\left[ M^{-1}\left( {c_\alpha}/{\sqrt{2\pi}} \right) \right]^2 \right),
\end{align*}
and $M = \phi/G$ is the Mills ratio. Thus, $E_{n}^{\alpha}$ lead to strictly better inference than $E_{n}^{\infty}$ for any fixed and sufficiently small level $\delta$. Observe that $\tau(6.1) \approx 0.1$, and so Bentkus-type e-values with $\alpha \in [0, 6.1)$ deliver better inference for fixed $\delta \leq 0.1$.


\begin{remark}[The $\alpha = 0$ case]
\label{remark:alpha_zero}
Although $c_0 = 1$ means that $\alpha = 0$ achieves the exact lower bound in~\eqref{eq:u_sandwich}---the tightest possible for any $\alpha$---this pointwise optimality does not necessarily translate into practical advantages when inference must cover a range of levels $\delta$. The reason is that the $\alpha = 0$ e-value is binary in $Z_n(\theta)$: once $Z_n(\theta) \geq \lambda$, the e-value equals $1/G(\lambda)$. Furthermore, the threshold function satisfies $U_{\delta,0}(\lambda) \in \{\lambda, \infty\}$, so a mixture of $\alpha = 0$ e-values degrades sharply as $\lambda$ moves away from its optimum, unlike the smooth and convex $U_{\delta,\alpha}$ for $\alpha \geq 1$ (see Proposition~\ref{proposition:threshold_convexity}). 
We will see a practical consequence of this behavior in 
Section~\ref{section:multiple_testing}. 
\end{remark}

What if $\delta$ is not fixed? As we discussed in Section~\ref{sec:motivation}, part of the power of (asymptotic) e-values is their ability to handle data-dependent significance levels $\delta$. But doing so has tradeoffs. In particular, $\lambda$ cannot be chosen as a function of $\delta$ as was done in the analysis above~\citep{chugg2026post}. There are two strategies for handling this: ex-ante anchoring and the method of mixtures.  
The next two sections show how Bentkus-type asymptotic e-values can improve performance under both of these strategies. 

\subsection{Ex-ante anchoring} \label{section:gaussian_evalues}

When constructing post-hoc confidence intervals with asymptotic e-values, the e-value may not depend on the significance level $\delta$, which may change as a function of the data. This implies that the analysis done in Section~\ref{section:near_optimal_asymptotic_e_values} does not hold. In particular, one cannot optimize $U_{\delta,\alpha}(\lambda)$ over $\lambda$, as the optimal value will depend on $\delta$. A natural strategy is to fix some data-independent $\delta_0$ and optimize $\lambda$ over $U_{\delta_0,\alpha}$. This is called \emph{ex-ante anchoring}~\citep{chugg2026post}, and $\delta_0$ is called the \emph{anchor}.

Let $\lambda_{0,\alpha}$ be the parameter that minimizes $U_{\delta_0,\alpha}$. Thus, $U_{\delta_0,\alpha}(\lambda_{0,\alpha})$ is the tightest lower boundary on our one-sided confidence interval (which we recall is $\{\theta: Z_n(\theta) < U_{\delta,\alpha}(\lambda)\}$) if $\delta = \delta_0$. As $\delta$ moves away from $\delta_0$, $U_{\delta,\alpha}(\lambda_0)$ grows and our confidence interval becomes looser. Figure~\ref{fig:u_vs_a} plots how fast this function grows for various values of $\alpha$ and $\delta$. 
Note that  the minima of $U_{\delta, \alpha}$ with $\alpha < \infty$ are smaller than the minima of $U_{\delta, \infty}$, and the functions $U_{\delta, \alpha}$ are above the function $U_{\delta, \infty}$ for big enough $\lambda$. Similarly, the minima of $U_{\delta, \alpha}$ are smaller for smaller values of $\alpha$, and the function $U_{\delta, \alpha}$ is above the function $U_{\delta, \alpha'}$ for $\alpha < \alpha'$ and small enough $\lambda$.

\begin{figure}[h!] 
    \center \includegraphics[width=\textwidth]{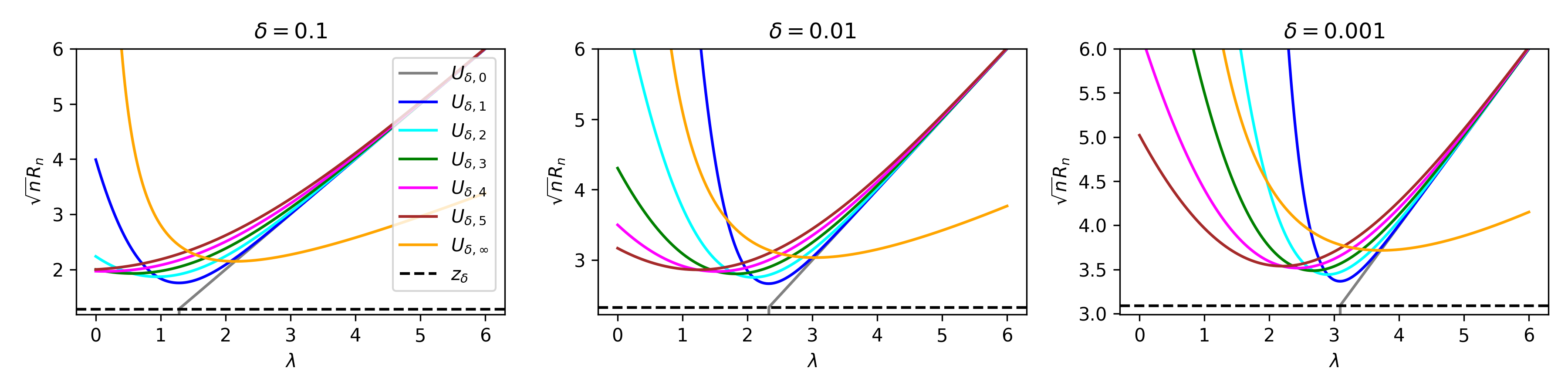}
  \caption{Evaluation of $U_{\delta, \infty}(\lambda)$  and $U_{\delta, \alpha}(\lambda)$ for $\alpha \in \{0, 1, 2, 3, 4, 5 \}$ and $\delta \in \{0.1, 0.01, 0.001\}$.} 
  \label{fig:u_vs_a}
\end{figure}

This should be interpreted as follows: Exponential e-values ($\alpha = \infty$) and large-$\alpha$ Bentkus-type e-values are more robust but also more conservative. The width of the interval is smaller for values $\lambda$ that are far from the minimizers, but larger otherwise. Consequently, the best choice of asymptotic e-value depends on the testing problem at hand. If the practitioner does not have a good sense of the (eventual) data-dependent level, more conservative e-values---higher values of $\alpha$---may be desirable.  

Figure~\ref{fig:u_vs_a} and Table~\ref{table:missing_factor} also illustrate the near-optimality of our Bentkus-type asymptotic e-values. In particular, the minima of $U_{\delta, \alpha}$ for $\alpha < \infty$ are smaller than the minima of $U_{\delta, \infty}$, a manifestation of the missing factor.  E.g., Table~\ref{table:missing_factor} gives $15.7\%$ and $12.2\%$ lower thresholds for $\alpha=1$ than for the exponential at $\delta=0.05$ and $0.01$. 
While explicitly establishing the minimizer of $U_{\delta,\alpha}$ is theoretically challenging, the following proposition proves that these threshold functions are convex. 
Thus, the minimizer of $U_{\delta, \alpha}$ can be obtained with arbitrary precision via convex optimization. The proof may be found in Appendix~\ref{proof:threshold_convexity}.

\begin{proposition}[Convexity of the ex-ante threshold function]
\label{proposition:threshold_convexity}
For any $\alpha \in [1, \infty)$ and target level $\delta \in (0, 1]$, the oracle threshold function $U_{\delta, \alpha}(\lambda)$ is strictly convex with respect to $\lambda$.
\end{proposition}

\subsection{Method of mixtures}
\label{sec:mixtures}

A second way to choose $\lambda$ is the \emph{method of mixtures}. 
Let $\pi$ be a prior distribution defined on a compact interval $\Lambda = [\lambda_{min}, \lambda_{max}] \subset \mathbb{R}$ and define the mixture asymptotic e-value as
\begin{align*}
E_n^{\alpha}(Z_n(\theta); \pi) = \int_\Lambda E_n^\alpha(Z_n(\theta); \lambda) \pi(\lambda) d\lambda = \int_\Lambda \frac{(Z_n(\theta) - \lambda)_+^\alpha}{I_\alpha(\lambda)} \pi(\lambda) d\lambda.
\end{align*}

One may wonder how the mixture e-value compares to the oracle ex-ante e-value, defined as $E_n^{\alpha}(\theta; \lambda_{\delta, \alpha})$, where $\lambda_{\delta, \alpha}$ is the optimal $\lambda$ if the level $\delta$ was known in advance. To quantify the cost of mixing, let $Z_{opt}(\delta) = \min_{\lambda \in \Lambda} U_{\delta, \alpha}(\lambda)$ denote the oracle ex-ante threshold, representing the tightest possible boundary achievable if the optimal $\lambda \in \Lambda$ were known in advance. Let $Z_{\pi}(\delta)$ be the data-dependent mixture boundary satisfying $E_n^{\alpha}(Z_{\pi}(\delta); \pi) = 1/\delta$. The following theorem establishes that, while the method of mixtures inherently incurs a hedging penalty against an optimal oracle, the threshold regret $\mathcal{R}(\delta)$ is strictly bounded by a term proportional to $1/G^{-1}(\delta)$. In particular, the theoretical cost of remaining agnostic to the data-dependent target level vanishes asymptotically as $\delta \to 0$. Its proof is deferred to Appendix~\ref{proof:mixture_regret}. 


\begin{theorem}[Threshold regret of mixture Bentkus e-values]
\label{prop:mixture_regret}
Fix $0<\alpha<\infty$. Assume $\Lambda$ is a compact interval of positive length and $\pi$ is a probability density with $\pi(\lambda)\geq\pi_{min}>0$ on $\Lambda$. There exists $\kappa\in(0,1]$, depending only on $\pi,\Lambda,\alpha$, such that for $0<\delta<1$,
\begin{align}\label{eq:upper_bound_zpi}
 Z_\pi(\delta)\leq Z_{opt}(\kappa\delta).
\end{align}
If also $\delta<1/2$ and $\Lambda$ contains a global minimizer of $U_{\kappa\delta,\alpha}$, then the regret
$R(\delta) = Z_\pi(\delta)-Z_{opt}(\delta)$ is upper bounded by $\log(c_\alpha/\kappa)/G^{-1}(\delta)$.
\end{theorem}


In many scenarios (such as the ones discussed in Section~\ref{section:applications}), one works with a range of potential data-dependent levels $\Delta \subset [\delta_{\min}, \delta_{\max}]$. Note that there exist $\lambda_{\min}$ and $\lambda_{\max}$ that are optimal for $\delta_{\min}$ and $\delta_{\max}$, respectively.  In such a case, $\lambda_\delta = \arg\min_\lambda U_{\delta, \alpha}(\lambda)\in[\lambda_{\min}, \lambda_{\max}]$. Thus, it suffices to consider a $\Lambda = [\lambda_{\min}, \lambda_{\max}]$, and $\pi$ putting mass on all $\Lambda$, in order to recover the theoretical guarantees established in Theorem~\ref{prop:mixture_regret}. It is also possible to obtain similar regret guarantees for a grid $\pi$, which we later use in experiments, as explained in Appendix~\ref{section:finite_grid}.

\paragraph{Choosing $\alpha$.}
For a narrow anticipated level range, $\alpha=1$ is a useful default; $\alpha=2$ trades slightly larger oracle thresholds for more tolerance to tuning mismatch. Larger $\alpha$, including the exponential limit, can suit broader ranges. The binary $\alpha=0$ case is most suitable near a fixed level. One can prespecify $\alpha$ and a mixture by minimizing worst-case threshold regret over an anticipated range $\Delta$, or mix over several $\alpha$'s. These choices must precede observing the data and eventual level; see Appendix~\ref{section:extensions}.

More generally, for an anticipated range $\Delta$, one can choose among a prespecified collection of families and mixtures by minimizing
\[
 \inf_{\alpha,\pi}\ \sup_{\delta\in\Delta}
 \{Z_{\pi,\alpha}(\delta)-G^{-1}(\delta)\}.
\]
A finite mixture across values of $\alpha$ also preserves validity. This is a design criterion, not a claim that one choice is universally optimal. Importantly, these choices cannot be optimized using the same data.

\section{Applications} \label{section:applications}

\subsection{Post-hoc inference} \label{section:posthoc_inference}

\begin{figure}[t] 
    \center \includegraphics[width=\textwidth]{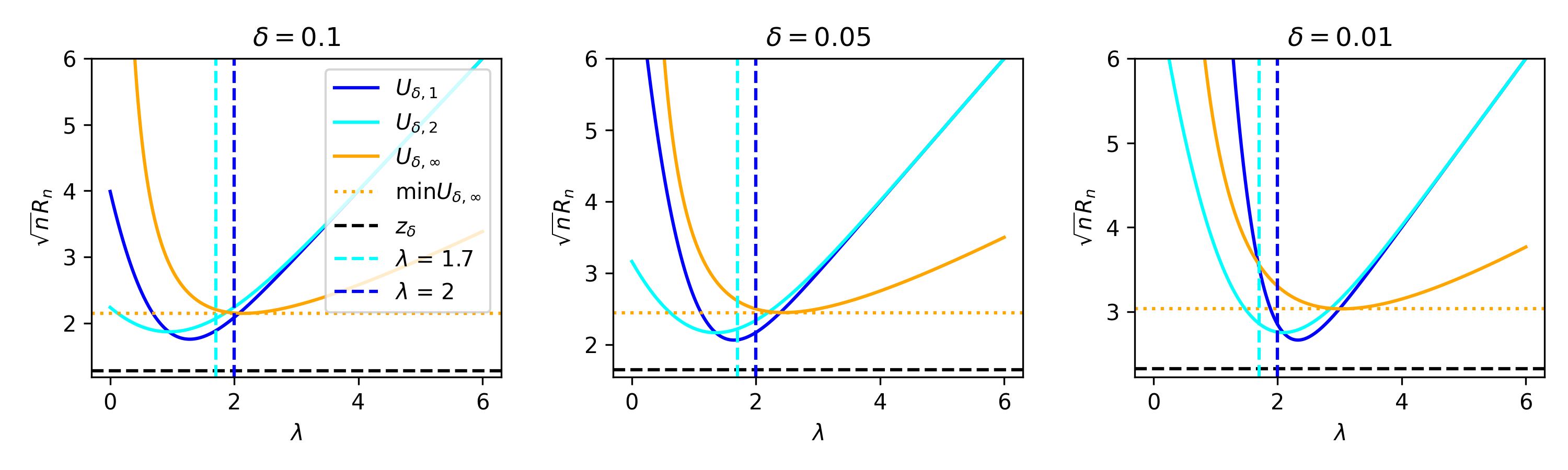}
    \caption{Threshold functions $U_{\delta,\alpha}(\lambda)$ for $\alpha \in \{1,2,\infty\}$ and $\delta \in \{0.1, 0.05, 0.01\}$.
The minima for $\alpha=1$ and $\alpha=2$ lie below the minimum for $\alpha=\infty$ at every displayed level, reflecting the near-optimality established in
Theorem~\ref{proposition:near_optimality}.}
  \label{fig:posthoc}
\end{figure}

Recall from Section~\ref{sec:motivation} that a post-hoc confidence set is obtained by inverting a family of e-values: $C_{n,\delta} = \{\theta: E_n(\theta) < 1/\delta\}$. 
As we saw in Section~\ref{section:gaussian_evalues}, exponential e-values have a better worst case performance as $\delta$ departs from the anchor $\delta_0$. Therefore, if a practitioner wants to remain completely agnostic as to the data-dependent level $\delta$, exponential-based e-values may be preferable. 

In many applications, however, complete agnosticism about $\delta$ is unnecessary. A practitioner may only care about a small collection of scientifically meaningful
levels, such as $\delta \in \{0.1,0.05,0.01\}$. Figure~\ref{fig:posthoc} shows that, over this
range, the Bentkus-type e-values $E^{1}_n(\mu;2)$ and $E^{2}_n(\mu;1.7)$ yield smaller thresholds than the best exponential e-value, even if the exponential
tuning parameter were chosen separately for each level. This illustrates that the fixed-level gains from Theorem~\ref{proposition:near_optimality} remain practically relevant at commonly used levels.


Figure~\ref{fig:posthoc} also implies that mixtures of $E_{n}^{1}(\theta; \lambda)$, $E_{n}^{2}(\theta; \lambda)$ will outperform mixtures of $E_{n}^{\infty}(\theta; \lambda)$. More precisely, consider defining a mixture distribution $\pi_\alpha$ which is a uniform distribution over the three values of $\lambda$ that are optimal for $\delta\in\{0.1,0.05, 0.01\}$. 
 Table~\ref{table:posthoc} reports the resulting e-values for various observed values of $Z$.    
 The Bentkus-type mixtures with $\alpha=1$ and $\alpha=2$ yield larger e-values than the exponential mixture across each value of $Z$. 
 For instance, when $Z=2.1$, the exponential mixture gives
$1/E_n=0.128$, so the null cannot be excluded at level $0.1$; by contrast, the
$\alpha=1$ and $\alpha=2$ mixtures give values of $0.077$ and $0.092$, respectively. Similarly, at $Z=2.4$ both Bentkus-type mixtures cross $0.05$, while the exponential mixture does not. At $Z=3.0$, only $\alpha=2$ crosses $0.01$: with $\alpha=1$ being slightly above $0.01$ (despite rounding to $0.010$ in the table).
 The experimental details can be found in Appendix~\ref{section:posthoc_experiment}.

\begin{table}[t]
\centering
\caption{
Post-hoc inference comparison between Bentkus-type and exponential mixture e-values. For each family, the mixture is uniform over the three values of $\lambda$ that are optimal for the target levels $\delta \in \{0.1,0.05,0.01\}$. Entries report both the resulting e-values and their reciprocals $1/E_n$. The $\alpha=1$ and $\alpha=2$ Bentkus-type mixtures produce smaller post-hoc levels than the exponential mixture $(\alpha=\infty)$ across the displayed values of $Z$.
}
\label{table:posthoc}
\begin{tabular}{rrrrrrrrr}
\toprule
$Z$ & $E_n^{0}$ & $E_n^{1}$ & $E_n^{2}$ & $E_n^\infty$ & $\delta(E_n^{0})$ & $\delta(E_n^{1})$ & $\delta(E_n^{2})$ & $\delta(E_n^\infty)$ \\
\midrule
2.1 & 10.00 & 13.02 & 10.87 & 7.81 & 0.100 & 0.077 & 0.092 & 0.128 \\
2.4 & 43.33 & 27.16 & 26.61 & 16.53 & 0.023 & 0.037 & 0.038 & 0.060 \\
2.7 & 43.33 & 63.57 & 58.51 & 35.37 & 0.023 & 0.016 & 0.017 & 0.028 \\
3.0 & 43.33 & 99.98 & 106.56 & 76.58 & 0.023 & 0.010 & 0.009 & 0.013 \\
\bottomrule
\end{tabular}

\end{table}

\subsection{Multiple testing}
\label{section:multiple_testing}

We next consider the consequences of Bentkus-type asymptotic e-values for multiple testing. Recall from Section~\ref{sec:motivation} that the e-BH procedure rejects the $k^*$ largest e-values, where $k^* = \max\{k: kE_{(k)}/K\geq 1/\delta^*\}$, and $E_{(1)}\geq E_{(2)}\geq \dots\geq E_{(K)}$ are the ordered asymptotic e-values (here we are hiding their dependence on the sample size).

Although the target FDR level $\delta^\star$ is fixed, the threshold applied to any particular e-value is data-dependent. Indeed, if $R_i$ denotes the rank of
$E_i$ among $E_1,\dots,E_K$, with $R_i=1$ corresponding to the largest e-value, then hypothesis $k$ can be rejected only if $E_{(k)}\geq K/(R_k \delta^*)$. Equivalently, the e-value $E_k$ is being tested at the data-dependent level $\delta_k := R_k\delta^*/K$. Thus, multiple testing provides another setting in which the relevant significance level is not known in advance.

This observation suggests a natural way to tune mixtures of Bentkus-type e-values. If the e-BH procedure is expected to reject at most a $\rho$-fraction of the hypotheses, then the ranks of rejected hypotheses lie in
$\{1,\dots,\rho K\}$, and the effective level lies in $[\delta^*/K, \rho\delta^*]$. 
In sparse large-scale testing problems, this range can be substantially smaller
than $(0,\delta^\star]$. For example, e-BH rejection rates are often on the order
of a few percent in practice, and sparse signals are common in applications such as neuroimaging~\citep{genovese2002thresholding, chumbley2009false} (active regions tend to comprise less than 10\% of the voxels) and genome-wide association studies~\citep{storey2003statistical} (the proportion of truly associated SNPs is often estimated to be well below $1\%$).  It is therefore natural to place mixture mass on values of the tuning parameter $\lambda$ that are optimized for levels in restricted ranges. 

\begin{table}[t]
\centering
\caption{
Multiple testing simulation comparing Bentkus-type and exponential mixture e-values under e-BH. Each mixture is uniform over 10 values of $\lambda$, chosen on a grid
optimized for data-dependent levels in $[\delta^*/K,\rho\delta^*] = [0.001,0.02]$, with $K=100$, $\delta^\star=0.1$, and $\rho=0.2$. Entries give the average number of rejections as the
proportion of non-null hypotheses varies. For each scenario, the highest number of rejections is highlighted in {\color{blue}blue}, while the most conservative method yielding the fewest discoveries is highlighted in {\color{red}red}: the $\alpha=1$ and $\alpha=2$ mixtures uniformly improve over the exponential mixture $(\alpha=\infty)$.
}
\label{table:multiple_testing}
\begin{tabular}{lrrrrr}
\toprule
Prop. Non-Null & 0.01 & 0.025 & 0.05 & 0.075 & 0.1 \\
\midrule
$\alpha = 0$  & {\color{red}0.00} & {\color{red}0.00} & {\color{red}1.78} & 4.72 & {\color{blue}7.64} \\
$\alpha = 1$  & 0.50 & {\color{blue}1.14} & {\color{blue}3.14} & {\color{blue}5.05} & 7.48 \\
$\alpha = 2$  & {\color{blue}0.54} & 1.12 & 3.02 & 4.99 & 7.40 \\
$\alpha = \infty$ & 0.50 & 0.91 & 2.55 & {\color{red}4.30} & {\color{red}6.32} \\
\bottomrule
\end{tabular}

\end{table}

\begin{remark}
    For multiple testing, the p-value based BH procedure controls FDR under independence and BY under arbitrary dependence when supplied with valid p-values, and either can be more powerful than e-BH. For post-hoc inference, testing a valid p-value at $\delta / L$ controls risk when $\delta$ is selected from $L$ prespecified levels. However, our prespecified asymptotic e-values allow arbitrary data-dependent levels with asymptotic risk control.
\end{remark}

We illustrate this idea in a simulation with $K=100$ hypotheses and target FDR level $\delta^\star=0.1$. We take $\rho=0.2$, so that the mixture is tuned for
levels in the interval $[\delta^*/K,\rho\delta^*]=[0.001,0.02]$. For each e-value family, we form a uniform mixture over a grid of 10 values of
$\lambda$, linearly spaced between the values that are optimal for the endpoints of this interval.

\paragraph{Gaussian observations.} We generate null statistics from $N(0,1)$ and non-null
statistics from $N(3.5,1)$, and apply e-BH to the resulting mixture e-values. Precise experimental configurations are detailed in Appendix~\ref{section:multiple_testing_experiment}. Table~\ref{table:multiple_testing} reports the results. The $\alpha=1,2$ mixtures give at least as many total rejections as the exponential mixture in every displayed scenario, with larger gains at moderate non-null proportions. 
The $\alpha=0$ mixture illustrates the tradeoff in Remark~\ref{remark:alpha_zero}: its stepwise, bounded evidence is ineffective for very sparse alternatives, but competitive in denser settings.

\paragraph{Heavy-tailed observations.}
We also generate raw samples with unit-variance $t_3/\sqrt3$ noise, which has infinite third absolute moment, for $n\in\{50,200,1000\}$ and five non-nulls with means $3.5/\sqrt n$. Table~\ref{table:heavytail} gives all results and Monte Carlo uncertainty. Using the same ten-point mixtures, at $n=200$ the $\alpha=1$ method yields $3.295$ true discoveries versus $2.815$ for the exponential, with empirical FDR $0.014$ versus $0.006$. 

\section{Conclusion}

This work connects the theory of 
near-optimal finite-sample concentration inequalities with asymptotic inference based on e-values. 
We introduced Bentkus-type asymptotic e-values, a novel family of statistics built upon $\alpha$-powered positive functions rather than exponentials. 
Under Gaussian domains of attraction, these statistics inherit the near-optimal behavior of Bentkus' original nonasymptotic bounds, and bypass the  ``missing factor'' inherent in existing asymptotic e-values. More concisely:  we have  provided an alternative to existing exponential e-values that lead to strictly tighter inference in common settings.

We also showed that these gains are not merely formal, demonstrating how our results translate into
practical advantages in two applications. In particular, they result in greater power when combined
with the e-BH procedure in multiple testing and in tighter asymptotic post-hoc confidence intervals.

\paragraph{Limitations.}
Our theory concerns scalar inference from i.i.d. observations in the Gaussian domain of attraction, with no finite-sample validity or convergence-rate guarantee. Validity is pointwise in the data distribution; we do not establish moderate-deviation power approximations or growing-$K$ FDR control without additional conditions. An extension to regression would require uniform integrability with estimated nuisance quantities; multivariate and dependent-data extensions also remain open. The simulations cover idealized Gaussian statistics and one heavy-tailed noise model, not general practical performance. 


\begin{ack}
We thank Neil Xu and Arun Kuchibhotla for helpful conversations. BC and AR acknowledge support from NSF grants IIS-2229881 and DMS-2310718.
\end{ack}

\bibliographystyle{apalike}
\bibliography{bib}


\appendix

\section{Proofs} \label{section:proofs}

\subsection{Proof of Theorem~\ref{theorem:main_theorem}} \label{proof:main_theorem}

Under the domain of attraction of a Gaussian assumption, the statistic $Z_n(\theta)$ is asymptotically normal as a simple consequence of Raikov's theorem, as noted e.g. in  \citet{maller1981theorem} and \citet{gine1997student}. Since the function
\begin{align*}
    x\mapsto E_{\lambda, \alpha}(x) := \frac{\lp x-\lambda \rp_+^\alpha}{I_\alpha(\lambda)}, \quad \alpha \in (0, \infty)
\end{align*}
is continuous on $\R$, and
\begin{align*}
    x\mapsto E_{\lambda, 0}(x) := \frac{\lp x-\lambda \rp_+^0}{I_0(\lambda)} = \frac{\mathbf{1} \lc x \geq \lambda \rc}{G(\lambda)}
\end{align*}
is continuous on $\R \backslash \{ \lambda \}$, the continuous mapping theorem also yields $E_{n}^{\alpha}(\theta; \lambda) \to E_{\lambda, \alpha}(Z)$ in distribution, where $Z \sim N(0,1)$.

For $\alpha=0$ uniform integrability follows from boundedness by $1/G(\lambda)$. For $\alpha>0$, we claim that $E_{n}^{\alpha}(\theta; \lambda)$ is uniformly integrable. It suffices to show that $(Z_n(\theta) - \lambda)_+^\alpha$ is uniformly integrable, as $I_\alpha(\lambda)$ is constant across $n$. We observe that $(1 + u/\alpha)_+^\alpha \leq e^u$ for $\alpha \in [0, \infty)$, and so
\begin{align*}
    (Z_n(\theta) - \lambda)_+^\alpha &= \alpha^{\alpha} \lp \frac{Z_n(\theta) - \lambda}{\alpha}\rp_+^\alpha
    \leq \alpha^{\alpha} \lp 1 + \frac{Z_n(\theta) - \lambda}{\alpha}\rp_+^\alpha
    \\&\leq \alpha^{\alpha} \exp \lp Z_n(\theta) - \lambda \rp.
\end{align*}
Thus, it suffices to show that $\exp \lp Z_n(\theta) \rp$ is uniformly integrable, which follows from \citet[Lemma C.1]{chugg2026post}. Finally, we invoke \citet[Theorem 25.12]{billingsley1995proba}, which shows that convergence in distribution together with uniform integrability proves that the means convergence, i.e., 
\[\E[E_n^\alpha(\theta;\lambda)] \stackrel{n\to\infty}{\to}\E[E_{\lambda,\alpha}(Z)] = 1,\]
completing the proof.

\subsection{Proof of Lemma~\ref{lemma:i_recursion}} \label{proof:i_recursion}

For the base cases, we have $I_0(\lambda) = \mathbb{E}[ \mathbf{1}\{Z \ge \lambda\} ] = \mathbb{P}(Z \ge \lambda) = G(\lambda)$. For $I_1(\lambda)$, using the property $\phi'(z) = -z\phi(z)$, we obtain
\begin{align*}
I_1(\lambda) = \int_\lambda^\infty (z-\lambda) \phi(z) dz = \int_\lambda^\infty z \phi(z) dz - \lambda \int_\lambda^\infty \phi(z) dz = \phi(\lambda) - \lambda G(\lambda).
\end{align*}
For the recurrence with $\alpha \ge 2$, using Gaussian integration by parts and the identity $z\phi(z) = -\phi'(z)$, we have
\begin{align*}
I_\alpha(\lambda) &= \int_\lambda^\infty (z-\lambda)^{\alpha-1} (z - \lambda) \phi(z) dz = \int_\lambda^\infty (z-\lambda)^{\alpha-1} z \phi(z) dz - \lambda I_{\alpha-1}(\lambda) \\
&= \big[-(z-\lambda)^{\alpha-1} \phi(z) \big]_\lambda^\infty + \int_\lambda^\infty (\alpha-1) (z-\lambda)^{\alpha-2} \phi(z) dz - \lambda I_{\alpha-1}(\lambda) \\
&= (\alpha-1) I_{\alpha-2}(\lambda) - \lambda I_{\alpha-1}(\lambda),
\end{align*}
where the boundary term vanishes because $\alpha \ge 2$ and the Gaussian density $\phi(z)$ decays exponentially.

\subsection{Proof of Theorem~\ref{proposition:near_optimality}} \label{proof:near_optimality}

We will prove the result by first handling the boundary case $\alpha = 0$, and then establishing the upper and lower bounds separately for $\alpha > 0$.

\paragraph{Case $\alpha = 0$.}
The thresholding condition $E_n^{0}(\theta;\lambda) \ge 1/\delta$ is satisfied if and only if $Z_n(\theta) \ge \lambda$ and $G(\lambda) \le \delta$. This requires setting $\lambda \ge G^{-1}(\delta)$. Consequently, the minimum threshold is exactly $\inf_\lambda U_{\delta, 0}(\lambda) = G^{-1}(\delta)$. Since $c_0 = 1$, the theoretical bounds $G^{-1}(\delta) \le \inf_\lambda U_{\delta, 0}(\lambda) \le G^{-1}(\delta/1)$ collapse to a single point, satisfying the theorem with exact equality.

\paragraph{Case $\alpha \in (0, \infty)$.}
For the upper bound, we define the transformed survival function as the infimum over the $\alpha$-powered positive functions
\begin{equation*}
    G_\alpha(\eta) := \inf_{\lambda < \eta} \mathbb{E} f_{\lambda,\eta,\alpha}(Z) = \inf_{\lambda < \eta} \frac{\mathbb{E}[(Z-\lambda)_+^\alpha]}{(\eta-\lambda)^\alpha} = \inf_{\lambda < \eta} \frac{I_\alpha(\lambda)}{(\eta-\lambda)^\alpha}.
\end{equation*}
\citet[Lemma 1.1]{bentkus2006domination} establishes that $G_\alpha(\eta) \le c_\alpha G^o(\eta)$ for $\alpha \in (0, \infty)$, where $G^o$ is the log-concave hull of $G$. Because the normal distribution is log-concave, $G = G^o$, and we have $G_\alpha(\eta) \le c_\alpha G(\eta)$.

We wish to evaluate this bound at the target quantile $\eta^* = G^{-1}(\delta/c_\alpha)$. Because $c_\alpha > 1$ for all $\alpha \in (0, \infty)$ by Lemma~\ref{lem:c_alpha_increasing}, we have $\delta/c_\alpha \in (0, 1)$ for any valid significance level $\delta \in (0, 1]$, ensuring the inverse survival function is well-defined. Substituting $\eta^*$ into the bound yields
\begin{equation*}
    G_\alpha(\eta^*) \le c_\alpha G(G^{-1}(\delta/c_\alpha)) = c_\alpha \left(\frac{\delta}{c_\alpha}\right) = \delta.
\end{equation*}

For any $\epsilon > 0$, strict monotonicity of $G$ gives $G_\alpha(\eta^*+\epsilon) \le c_\alpha G(\eta^*+\epsilon) < \delta$. By the definition of the infimum, there exists a shift parameter $\lambda_\epsilon < \eta^*+\epsilon$ such that
\[
\frac{I_\alpha(\lambda_\epsilon)}{(\eta^*+\epsilon-\lambda_\epsilon)^\alpha} < \delta \iff U_{\delta,\alpha}(\lambda_\epsilon) < \eta^*+\epsilon.
\]
Thus $\inf_{\lambda\in\mathbb{R}} U_{\delta,\alpha}(\lambda) < \eta^*+\epsilon$ for every $\epsilon>0$. Letting $\epsilon\downarrow0$ proves the upper bound
\[
\inf_{\lambda\in\mathbb{R}}U_{\delta,\alpha}(\lambda) \le \eta^* = G^{-1}(\delta/c_\alpha).
\]

\textbf{Lower Bound.} We recall that $E_n^\alpha(\theta;\lambda) \ge 1/\delta$ holds if, and only if,
\begin{equation*}
    Z_n(\theta) \ge U_{\delta,\alpha}(\lambda) := \lambda + \left(\frac{I_\alpha(\lambda)}{\delta}\right)^{\frac{1}{\alpha}}.
\end{equation*}
Furthermore, if $Z_n(\theta) \sim N(0,1)$ then $E_n^\alpha(\theta;\lambda)$ is an exact e-value valid in finite samples. Thus, by Markov's inequality, $\mathbb{P}(E_n^\alpha(\theta;\lambda) \ge 1/\delta) \le \delta$. Consequently $G(U_{\delta,\alpha}(\lambda))\leq\delta$ for every fixed $\lambda$, so $U_{\delta,\alpha}(\lambda)\geq G^{-1}(\delta)$ for every $\lambda$. Taking the infimum proves the lower bound.

\subsection{Proof of Proposition~\ref{proposition:threshold_convexity}} \label{proof:threshold_convexity}

Let us recall that
\begin{align*}
U_{\delta, \alpha}(\lambda) &= \lambda + \delta^{-1/\alpha} \big( I_\alpha(\lambda) \big)^{1/\alpha}.
\end{align*}
Because the first term $\lambda$ is linear, it suffices to prove that the mapping $\lambda \mapsto I_\alpha(\lambda)^{1/\alpha}$ is convex. We can recognize this term as the $L_\alpha$-norm of the positive-part random variable, which we denote by
\begin{align*}
N_\alpha(\lambda) &= \| (Z - \lambda)_+ \|_\alpha = \left( \mathbb{E} \big[ (Z - \lambda)_+^\alpha \big] \right)^{1/\alpha}.
\end{align*}

Let $\lambda_1, \lambda_2 \in \mathbb{R}$ and $t \in [0, 1]$. We define the convex combination $\lambda_t = t \lambda_1 + (1 - t) \lambda_2$. Because the function $x \mapsto \max(0, x)$ is convex, applying it to the random variable $Z - \lambda_t$ yields the pointwise inequality
\begin{align*}
(Z - \lambda_t)_+ &= \big( t(Z - \lambda_1) + (1-t)(Z - \lambda_2) \big)_+ \\
&\le t(Z - \lambda_1)_+ + (1-t)(Z - \lambda_2)_+.
\end{align*}
Because the $L_\alpha$-norm is strictly increasing with respect to pointwise positive domination, taking the norm of both sides preserves the inequality and we obtain
\begin{align*}
N_\alpha(\lambda_t) &\le \big\| t(Z - \lambda_1)_+ + (1-t)(Z - \lambda_2)_+ \big\|_\alpha.
\end{align*}
Because $\alpha \ge 1$, we can apply Minkowski's inequality to bound the norm of the sum by the sum of the norms, yielding
\begin{align*}
\big\| t(Z - \lambda_1)_+ + (1-t)(Z - \lambda_2)_+ \big\|_\alpha &\le \big\| t(Z - \lambda_1)_+ \big\|_\alpha + \big\| (1-t)(Z - \lambda_2)_+ \big\|_\alpha \\
&= t N_\alpha(\lambda_1) + (1-t) N_\alpha(\lambda_2).
\end{align*}
This sequence of inequalities establishes that $N_\alpha(t \lambda_1 + (1-t)\lambda_2) \le t N_\alpha(\lambda_1) + (1-t)N_\alpha(\lambda_2)$, confirming that $N_\alpha(\lambda)$ is a convex function. Because $U_{\delta, \alpha}(\lambda)$ is the sum of a linear function and a positively scaled convex function, it is convex. Furthermore, because the standard normal distribution has full support on $\mathbb{R}$, the random variables $(Z - \lambda_1)_+$ and $(Z - \lambda_2)_+$ are not linearly dependent for $\lambda_1 \neq \lambda_2$. For $\alpha>1$, Minkowski's inequality is strict. For $\alpha=1$, the first positive-part inequality is strict on a set of positive Gaussian probability when $0<t<1$. Thus $U_{\delta,\alpha}$ is strictly convex in both cases.

\subsection{Proof of Theorem~\ref{prop:mixture_regret}} \label{proof:mixture_regret}

First, compact support ensures that the mixture is an asymptotic e-value: its integrand is bounded by $(Z_n(\theta)-\lambda_{\min})_+^\alpha/I_\alpha(\lambda_{\max})$, a uniformly integrable envelope by the proof of Theorem~\ref{theorem:main_theorem}. Convergence in distribution and uniform integrability therefore give convergence of the mixture expectation to one.

Let us first establish the duality between the oracle threshold and the unmixed e-value. For a fixed $\lambda \in \Lambda$, $U_{\delta,\alpha}(\lambda)$ uniquely solves $E_n^{\alpha}(z; \lambda) = 1/\delta$. Let $M(z) = \max_{s \in \Lambda} E_n^\alpha(z; s)$. Because $z \mapsto E_n^\alpha(z;s)$ is strictly increasing for $\alpha > 0$ and $z \geq s$, its upper envelope $M(z)$ is also strictly increasing for $z \geq \lambda_{min}$. Evaluating at the oracle minimum $Z_{opt}(\delta) = \min_{\lambda\in\Lambda} U_{\delta,\alpha}(\lambda)$ establishes the duality identity
\begin{equation}
    M(Z_{opt}(\delta)) = \frac{1}{\delta}.
    \label{eq:duality_identity}
\end{equation}
Because $M(z)$ is strictly increasing, it is invertible, mapping a target e-value level $1/u$ uniquely to the oracle threshold $Z_{opt}(u)$.

We prove the remainder of the result in three steps:
\begin{enumerate}
    \item We bound the data-dependent threshold to the active rejection domain.
    \item We establish a strictly positive mixing penalty $\kappa$ to bound the mixture threshold.
    \item We bound the threshold regret $\mathcal{R}(\delta)$ to conclude the proof.
\end{enumerate}

\textbf{Details of step 1.} We evaluate the mixture ratio on the active rejection domain. For any valid target level $\delta \in (0, 1]$, the data-dependent threshold $Z_\pi(\delta)$ is defined by the condition:
\begin{equation*}
    E_n^\alpha(Z_\pi(\delta); \pi) = \frac{1}{\delta}.
\end{equation*}
Because $\delta \le 1$, it immediately follows that $1/\delta \ge 1$, which implies:
\begin{equation}
    E_n^\alpha(Z_\pi(\delta); \pi) \ge 1.
    \label{eq:mixture_ge_1}
\end{equation}
Furthermore, the mixture e-value is a weighted average of individual e-values over a valid probability density $\pi(\lambda)$ (which integrates to $1$). It is thus strictly bounded above by the maximum envelope $M(z)$. For any $z \in \mathbb{R}$:
\begin{equation}
    E_n^\alpha(z; \pi) = \int_\Lambda E_n^\alpha(z; \lambda) \pi(\lambda) d\lambda \le \int_\Lambda M(z) \pi(\lambda) d\lambda = M(z).
    \label{eq:mixture_le_envelope}
\end{equation}
Evaluating \eqref{eq:mixture_le_envelope} at the specific threshold $z = Z_\pi(\delta)$ and chaining it with \eqref{eq:mixture_ge_1} yields $M(Z_\pi(\delta)) \ge 1$.

By the duality identity \eqref{eq:duality_identity}, we know that the oracle threshold for $\delta = 1$ perfectly satisfies $M(Z_{opt}(1)) = 1$. Substituting this into our inequality gives:
\begin{equation*}
    M(Z_\pi(\delta)) \ge M(Z_{opt}(1)).
\end{equation*}
Because the envelope $M(z)$ is a strictly increasing function on $z \geq \lambda_{min}$, and $Z_\pi(\delta) \geq \lambda_{min}$ (otherwise $E_n^\alpha(Z_\pi(\delta); \pi) = 0$), applying its inverse strictly preserves the direction of the inequality. We define $z_0 := Z_{opt}(1)$, establishing the formal lower bound for our threshold:
\begin{equation*}
    Z_\pi(\delta) \ge z_0.
\end{equation*}

\textbf{Details of step 2.} To bound the cost of mixing, we analyze the ratio of the mixture e-value to the oracle envelope:
\begin{equation*}
    H(z) = \frac{E_n^\alpha(z; \pi)}{M(z)}.
\end{equation*}
Because $M(z_0) = 1$ and $M$ is strictly increasing, the envelope is strictly positive for all $z \ge z_0$. Furthermore, for any fixed $z \ge z_0$, the normalized integrand $E_n^\alpha(z; \lambda)/M(z)$ is a continuous function of $\lambda$ that achieves a maximum of exactly $1$. Because $\pi(\lambda) \ge \pi_{\min} > 0$, the integral $H(z)$ is strictly positive and continuous for all $z \ge z_0$.

Next, we examine the asymptotic behavior of $H(z)$ as $z \to \infty$. For $\alpha > 0$, factoring out $z^\alpha$, we can write the e-value as $E_n^\alpha(z;\lambda) = z^\alpha(1-\lambda/z)_+^\alpha / I_\alpha(\lambda)$. As $z \to \infty$, the term $(1-\lambda/z)_+^\alpha$ converges to $1$ uniformly for all $\lambda \in \Lambda$. Therefore, for $\alpha > 0$, we can permute the limit and minimum to find the normalized integrand converges pointwise as
\[
\lim_{z\rightarrow\infty}\frac{E_n^\alpha(z;\lambda)}{M(z)} = \min_{s\in\Lambda}\lim_{z\rightarrow\infty}\frac{E_n^\alpha(z;\lambda)}{E_n^\alpha(z;s)} = \min_{s\in\Lambda}\frac{1/I_\alpha(\lambda)}{1/I_\alpha(s)} = \frac{I_\alpha(\lambda_{max})}{I_\alpha(\lambda)}.
\]
The minimum is achieved at $\lambda_{max}$ because $I_\alpha(\lambda)$ is a strictly decreasing function. By the dominated convergence theorem, the limit of the integral is
\begin{equation*}
    \lim_{z \to \infty} H(z) = \int_\Lambda \frac{I_\alpha(\lambda_{\max})}{I_\alpha(\lambda)} \pi(\lambda) d\lambda =: \kappa_\infty > 0.
\end{equation*}
Because $H(z)$ is continuous, strictly positive on the closed interval $[z_0, \infty)$, and converges to a strictly positive limit $\kappa_\infty$ as $z \to \infty$, its infimum on this domain is strictly positive. Let $\kappa = \inf_{z \ge z_0} H(z) > 0$. 

By definition, for all $z \ge z_0$, we have $E_n^\alpha(z; \pi) = M(z)H(z) \ge \kappa M(z)$. Substituting $z = Z_\pi(\delta)$ (which satisfies $z \ge z_0$ by Step 1) and noting that $E_n^\alpha(Z_\pi(\delta); \pi) = 1/\delta$, we obtain
\begin{equation*}
    M(Z_\pi(\delta)) \le \frac{1}{\kappa \delta}.
\end{equation*}
Applying the duality identity \eqref{eq:duality_identity} at the level $\kappa \delta$, we know $M(Z_{opt}(\kappa \delta)) = 1/(\kappa \delta)$. Substituting this yields $M(Z_\pi(\delta)) \le M(Z_{opt}(\kappa \delta))$. Because $M(z)$ is strictly increasing, applying its inverse preserves the inequality, establishing $Z_\pi(\delta) \le Z_{opt}(\kappa \delta)$.

\textbf{Details of step 3.} Assume $0<\delta<1/2$ and that $\Lambda$ contains a global minimizer at level $\kappa\delta$. Theorem~\ref{proposition:near_optimality} then gives $Z_{opt}(\kappa\delta)\leq G^{-1}(\kappa\delta/c_\alpha)$, while $Z_{opt}(\delta)\geq G^{-1}(\delta)$ holds even for a restricted oracle. The threshold regret $\mathcal{R}(\delta)$ is upper bounded as
\begin{equation}
    \mathcal{R}(\delta) = Z_\pi(\delta) - Z_{opt}(\delta) \le Z_{opt}(\kappa \delta) - Z_{opt}(\delta) \le G^{-1}(\kappa \delta / c_\alpha) - G^{-1}(\delta).
    \label{eq:regret_bound_initial}
\end{equation}
Let $x = \delta$ and $y = \kappa \delta / c_\alpha$. By the Fundamental Theorem of Calculus, noting that the derivative of $G^{-1}(u)$ is $-1/\phi(G^{-1}(u))$ where $\phi$ is the standard normal density, we have:
\begin{equation*}
    G^{-1}(y) - G^{-1}(x) = \int_y^x \frac{1}{\phi(G^{-1}(u))} du.
\end{equation*}
The standard Gaussian Mills ratio bound guarantees $\phi(t) > t G(t)$ for all $t > 0$. Substituting $t = G^{-1}(u)$ yields $\phi(G^{-1}(u)) > u G^{-1}(u)$. Therefore,
\begin{equation*}
    \int_y^x \frac{1}{\phi(G^{-1}(u))} du < \int_y^x \frac{1}{u G^{-1}(u)} du.
\end{equation*}
Because the inverse survival function is strictly decreasing, $G^{-1}(u) \ge G^{-1}(x)$ for all $u \in [y, x]$. Factoring this minimum out of the integral gives a strict upper bound:
\begin{equation*}
    \int_y^x \frac{1}{u G^{-1}(u)} du \le \frac{1}{G^{-1}(x)} \int_y^x \frac{1}{u} du = \frac{\ln(x/y)}{G^{-1}(x)}.
\end{equation*}
Substituting $x$ and $y$ back into the bound yields:
\begin{equation*}
    G^{-1}\left(\frac{\kappa \delta}{c_\alpha}\right) - G^{-1}(\delta) < \frac{\ln(c_\alpha / \kappa)}{G^{-1}(\delta)}.
\end{equation*}
Applying this to \eqref{eq:regret_bound_initial} establishes the final regret bound.

\section{Auxiliary results}

\begin{lemma}
\label{lem:c_alpha_increasing}
The function $c(\alpha)$ defined as $c(\alpha) = e^\alpha \alpha^{-\alpha} \Gamma(\alpha+1)$ for $\alpha > 0$, with $c(0) = 1$, is strictly increasing for all real $\alpha \ge 0$.
\end{lemma}

\begin{proof}
We first show that $c(\alpha)$ is strictly increasing for $\alpha > 0$ by establishing that its natural logarithm, $f(\alpha) = \ln c(\alpha)$, is strictly increasing. Taking the logarithm yields:
\begin{equation*}
    f(\alpha) = \alpha - \alpha \ln \alpha + \ln \Gamma(\alpha+1).
\end{equation*}

Differentiating with respect to $\alpha$, we obtain the first derivative:
\begin{equation*}
    f'(\alpha) = 1 - (\ln \alpha + 1) + \psi(\alpha+1) = \psi(\alpha+1) - \ln \alpha,
\end{equation*}
where $\psi(x) = \frac{d}{dx} \ln \Gamma(x)$ is the digamma function.

To determine the sign of $f'(\alpha)$, we examine the second derivative:
\begin{equation*}
    f''(\alpha) = \psi_1(\alpha+1) - \frac{1}{\alpha},
\end{equation*}
where $\psi_1(x) = \frac{d}{dx}\psi(x)$ is the trigamma function. The trigamma function admits the series representation $\psi_1(x) = \sum_{k=0}^\infty \frac{1}{(x+k)^2}$. Shifting the index by 1, we can write:
\begin{equation*}
    \psi_1(\alpha+1) = \sum_{k=1}^\infty \frac{1}{(\alpha+k)^2}.
\end{equation*}

Because the function $g(x) = \frac{1}{(\alpha+x)^2}$ is strictly decreasing for $x > 0$, we can strictly upper bound the infinite sum by the corresponding integral:
\begin{equation*}
    \sum_{k=1}^\infty \frac{1}{(\alpha+k)^2} < \int_0^\infty \frac{1}{(\alpha+x)^2} dx = \left[ -\frac{1}{\alpha+x} \right]_0^\infty = \frac{1}{\alpha}.
\end{equation*}

Substituting this bound back into the second derivative yields:
\begin{equation*}
    f''(\alpha) < \frac{1}{\alpha} - \frac{1}{\alpha} = 0.
\end{equation*}
Because $f''(\alpha) < 0$ for all $\alpha > 0$, the first derivative $f'(\alpha)$ is a strictly decreasing function. 

Next, we evaluate the asymptotic limit of $f'(\alpha)$ as $\alpha \to \infty$. Using the well-known asymptotic expansion of the digamma function, $\psi(\alpha+1) \sim \ln \alpha + \frac{1}{2\alpha} - \mathcal{O}\left(\frac{1}{\alpha^2}\right)$, we find:
\begin{equation*}
    \lim_{\alpha \to \infty} f'(\alpha) = \lim_{\alpha \to \infty} \left( \ln \alpha + \frac{1}{2\alpha} - \ln \alpha \right) = 0^+.
\end{equation*}

Since $f'(\alpha)$ is strictly decreasing and approaches $0$ from above as $\alpha \to \infty$, it must be that $f'(\alpha) > 0$ for all $\alpha > 0$. Because its derivative is strictly positive, $f(\alpha) = \ln c(\alpha)$ is strictly increasing, which implies that $c(\alpha)$ is strictly increasing for all real $\alpha > 0$.

Finally, we establish continuity at the boundary $\alpha = 0$ by evaluating the limit of $c(\alpha)$ as $\alpha \to 0^+$. We rewrite $c(\alpha)$ using the exponential function:
\begin{equation*}
    \lim_{\alpha \to 0^+} c(\alpha) = \lim_{\alpha \to 0^+} \Gamma(\alpha+1) \exp(\alpha - \alpha \ln \alpha).
\end{equation*}
Since $\Gamma(1) = 1$ and $\lim_{\alpha \to 0^+} (\alpha - \alpha \ln \alpha) = 0$, we obtain:
\begin{equation*}
    \lim_{\alpha \to 0^+} c(\alpha) = 1 \cdot e^0 = 1.
\end{equation*}
Because $\lim_{\alpha \to 0^+} c(\alpha) = 1 = c(0)$, the function is continuous at $\alpha = 0$. Since $c(\alpha)$ is continuous at $0$ and strictly increasing on $(0, \infty)$, we conclude that $c(\alpha)$ is strictly increasing for all real $\alpha \ge 0$.
\end{proof}

\begin{lemma} \label{lemma:fracional_moments}
    For any real $\alpha > 0$, it holds that
\begin{align} \label{eq:i_alpha_parabolic}
    I_\alpha(\lambda) = \frac{\Gamma(\alpha+1)}{\sqrt{2\pi}} e^{-\lambda^2/4} D_{-(\alpha+1)}(\lambda).
\end{align}
where 
\begin{align*}
    D_{-\nu}(z) = \frac{e^{-z^2/4}}{\Gamma(\nu)} \int_0^\infty x^{\nu-1} e^{-zx - x^2/2} dx.
\end{align*}
is the parabolic cylinder function for $\text{Re}(\nu) > 0$.
\end{lemma}

\begin{proof}


Recall the integral definition of the truncated and uncentered $\alpha$-moment of a standard Gaussian distribution
\begin{align*}
    I_\alpha(\lambda) &:= \mathbb{E}[(Z-\lambda)_+^\alpha] = \frac{1}{\sqrt{2\pi}} \int_\lambda^\infty (z-\lambda)^\alpha e^{-z^2/2} dz.
\end{align*}

Applying the change of variables $x = z - \lambda$, we obtain
\begin{align*}
    I_\alpha(\lambda) &= \frac{1}{\sqrt{2\pi}} \int_0^\infty x^\alpha e^{-(x+\lambda)^2/2} dx = \frac{e^{-\lambda^2/2}}{\sqrt{2\pi}} \int_0^\infty x^\alpha e^{-x^2/2 - \lambda x} dx.
\end{align*}

By setting $\nu = \alpha + 1$ and $z = \lambda$, we can perfectly match our integral to the representation
\begin{align*}
    \int_0^\infty x^\alpha e^{-\lambda x - x^2/2} dx = \Gamma(\alpha+1) e^{\lambda^2/4} D_{-(\alpha+1)}(\lambda).
\end{align*}

Substituting this identity back into our expression for $I_\alpha(\lambda)$ yields the closed-form analytical expression~\eqref{eq:i_alpha_parabolic}.



\end{proof}

\section{Evidence versus testing} \label{section:evidence_vs_testing}

Stepping back from the technicalities for a moment, one perspective to bring to bear on our results here is the age-old tension between evidence versus testing in statistics. The story is as old as modern frequentism itself: Fisher introduced the p-value as a notion of evidence against the null hypothesis, while Neyman and Pearson wanted to use the p-value for decision-making and hypothesis testing. The differences between the Fisherian viewpoint and that advocated by Neyman and Pearson manifests in what one chooses to do with the p-value resulting from a study, not in how one designs the p-value in the first place. Intriguingly, the same cannot be said in the study of e-values, where a focus on testing versus evidence leads to different design criteria. 

If one is focused solely on testing and wants to maximize a traditional notion of power, then the optimal e-values are ``all-or-nothing:'' they have the form $E = \mathbf{1}(A)/\delta$ for some event $A$. Intuitively, we don't care about maximizing the size of $E$, we just want to reject as soon as possible. We thus spread $E$'s mass around as much as possible while ensuring that $\E_P[E]\leq 1$. 
If one is concerned about evidence, on the other hand, then one cares about the growth of the e-value under the alternative hypothesis. This leads to the notion of \emph{e-power}, which is to maximize $\E[\log(E)]$, where the expectation is taken over the alternative.E-values that maximize this criterion usually take the form of exponential functions. We refer to \citet[Section 1]{ramdas2025hypothesis} for much discussion on different perspectives one can have on e-values.


For our purposes, neither a full evidentiary view nor a full testing view fit are entirely appropriate. Both post-hoc inference and multiple testing imply non-fixed level tests (as we elucidate in Section~\ref{section:applications}), so all-or-nothing e-values may not be optimal: if the eventual level $\delta'$ is $\delta - \epsilon$ for any arbitrarily small $\epsilon > 0$, the test will never reject. Conversely, we should not optimize for an abstract concept of evidence; instead, we should take into account that the e-values will be thresholded at levels $1/\delta'$, with $\delta'$ taking values in a sensible range.  For this reason, we build e-values using $\alpha$-powered positive functions that lie \textit{in between} the indicator and exponential functions, leading to more powerful tests in many applications of interest.

The e-power criterion is also motivated by products of e-values. If an e-value $E$ is a product of e-values $E_i$, i.e. $E = \prod_{i \leq n} E_i$, it converges to infinity exponentially fast under the alternative, and choosing the e-values that maximize $\E[\log E]$ also maximizes the rate at which the product approaches infinity. Hence, the logarithm arises naturally in reasoning about a sequence of
products. Interestingly, in the context of concentration inequalities, it is well known that  Cramér-Chernoff
bounds (that is, applying Markov’s inequality after taking the exp function of the sum of random variables) are
optimal among all tail bounds that exploit only the product structure of i.i.d. observations. Thus, exponential functions naturally manifest in both the concentration inequalities and e-value literature when manipulating products of random variables. However, near-optimal concentration inequalities avoid product structures by leveraging $\alpha$-powered positive functions. While product structures are inherent to exact e-values, asymptotic e-values rely on central limit theorems and uniform integrability, and hence are not reliant in product structures. This further motivates the study of $\alpha$-powered functions for asymptotic e-values. 

Finally, let us note that $\alpha$-powered positive asymptotic e-values lead to (at least theoretically) a strictly more general approach than exponential-based asymptotic e-values. To see this, for $\alpha > 0$, define the class of functions
\begin{align*}
    \F_+^{(\alpha)} := \lc f:f(u) = \int_{-\infty}^{\infty} (u - r)_+^\alpha \nu(dr) \text{ for some Borel measure } \nu \geq 0 \text{ on } \R \text{ and all } u \in \R\rc.
\end{align*}
The following fact from \citet{pinelis2006binomial} characterizes the class of such functions. 

\begin{fact} [Proposition 1.1 in \citet{pinelis2006binomial}]  \label{fact:characterization_alpha_functions}

For every natural $\alpha$, one has $f \in \F_+^{(\alpha)}$ iff $f$ has finite derivatives $f^{(0)}:= f, f^{(1)}:= f', \ldots, f^{(\alpha - 1)}$ on $\R$ such that $f^{(\alpha - 1)}$ is convex on $\R$ and $\lim_{u \to -\infty}f^{(j)}(u) = 0$ for $j = 0, 1, \ldots, \alpha - 1$.
\end{fact}

\begin{proof}
If $f \in \F^{(\alpha)}$ for $\alpha \in \Nb$, then $f'(u) =\alpha \int_{-\infty}^\infty (u - r)_+^{\alpha - 1} \nu (dr)$. Vice versa, if $f$ satisfies the conditions listed after ``iff'', then one can use the Fubini theorem repeatedly to see that for all real $u$
\begin{align*}
    f(u) &= \int_{-\infty}^u f'(r) dr = \int_{-\infty}^u dr \int_{-\infty}^r f''(s) ds = \int_{-\infty}^u (u - s) f''(s) ds = \ldots
    \\& = \int_{-\infty}^{u} \frac{(u-s)^{\alpha - 1}}{(\alpha  - 1)!}f^{(\alpha)}(s) ds = \int_{-\infty}^{u} \frac{(u-s)^\alpha}{\alpha!}df^{(\alpha)}(s) = \int_{-\infty}^{\infty} (u - s)_+^\alpha \nu(ds),
\end{align*}
where $f^{(\alpha)}$ is the (nondecreasing) right derivative of the convex function $f^{(\alpha - 1)}$ and $\nu(ds) := df^{(\alpha)}(s) / \alpha!$.
\end{proof}

For each positive integer $\alpha$ and $\lambda>0$, Fact~\ref{fact:characterization_alpha_functions} applies to $u\mapsto e^{\lambda u}$; its proof gives the representing measure $\nu(dr)=\lambda^{\alpha+1}e^{\lambda r}\,dr/\alpha!$. After normalizing the kernels as in~\eqref{eq:alpha_powered_e_value}, the exponential e-value becomes a probability mixture with measure
\[
\pi_\lambda(dr)=\frac{\lambda^{\alpha+1}}{\Gamma(\alpha+1)}e^{\lambda r-\lambda^2/2}I_\alpha(r)\,dr.
\]
The gamma integral verifies this representation for every real $\alpha>0$:
\[
\int_{\R} E_n^\alpha(\theta;r)\,\pi_\lambda(dr)
=e^{\lambda Z_n(\theta)-\lambda^2/2}=E_n^\infty(\theta;\lambda).
\]
Taking expectations with a standard Gaussian in place of $Z_n(\theta)$ and applying Tonelli's theorem shows that $\int_{\R}\pi_\lambda(dr)=1$. Thus, $\alpha$-powered positive e-values are somehow a refinement of the exponential e-value, as a mixture of the former allows for recovering the latter.

\section{Extensions and practical scope}\label{section:extensions}

Section~\ref{sec:bentkus-type} develops Bentkus-type asymptotic e-values for one-sided mean tests with Gaussian limits, and Section~\ref{section:applications} uses finite mixtures for post-hoc inference and multiple testing. This appendix addresses three questions arising from that scope: how to test for deviations in either direction (Section~\ref{section:two_sided}), which parts of the construction extend to other estimators and limiting distributions (Section~\ref{section:other_limits}), and how the continuous-mixture threshold guarantees translate to finite grids (Section~\ref{section:finite_grid}).

\subsection{Two-sided tests.}\label{section:two_sided}
The main paper tests $H_0:\mu\leq\theta$, so rejection is driven by deviations above $\theta$. When departures below $\theta$ are also relevant, one instead tests $H_0:\mu=\theta$ against $\mu\neq\theta$. As one would expect, averaging the e-values for the two directions gives the relevant two-sided test. 

Suppose the observations are i.i.d.\ with mean $\mu$ and lie in the Gaussian domain of attraction, and write $Z_n(\theta)=\sum_{i=1}^n(X_i-\theta)/\sqrt{\sum_{i=1}^n(X_i-\theta)^2}$. For fixed $\alpha\geq0$ and $\lambda\in\R$, let $E^\alpha(z;\lambda)=(z-\lambda)_+^\alpha/I_\alpha(\lambda)$, where $I_\alpha(\lambda)=\E[(Z-\lambda)_+^\alpha]$ for $Z\sim N(0,1)$. (Recall that  at $\alpha=0$, the numerator is interpreted as $\mathbf1\{z\geq\lambda\}$.) Define
\[
 E^{\mathrm{two}}_{n,\alpha}(\theta;\lambda)
 =\tfrac12\{E^\alpha(Z_n(\theta);\lambda)+E^\alpha(-Z_n(\theta);\lambda)\}.
\]
Apply Theorem~\ref{theorem:main_theorem} to both signs of the observations. Their average is an asymptotic e-value, and inversion gives a two-sided post-hoc confidence set. Write $U_{u,\alpha}(\lambda)=\inf\{z:E^\alpha(z;\lambda)\geq1/u\}$ for the one-sided rejection threshold, with $\inf\emptyset=\infty$. For $\lambda\geq0$, at most one positive-part term is active, apart from the immaterial $\alpha=\lambda=z=0$ boundary. Thus rejection at level $\delta\in(0,1)$ is equivalent to $|Z_n(\theta)|\geq U_{\delta/2,\alpha}(\lambda)$.
For unrestricted $\lambda$, the fixed-level oracle inequality from Theorem~\ref{proposition:near_optimality} applies to $\inf_{\lambda\in\R}U_{\delta/2,\alpha}(\lambda)$. To interpret it as a two-sided threshold one must ensure that a minimizing $\lambda$ is nonnegative; then
\[
 G^{-1}(\delta/2)\leq\inf_{\lambda\geq0}U_{\delta/2,\alpha}(\lambda)
 \leq G^{-1}(\delta/(2c_\alpha)).
\]
Here $G$ is the standard Gaussian survival function and $c_\alpha=e^\alpha\alpha^{-\alpha}\Gamma(\alpha+1)$ for $\alpha>0$, with $c_0=1$. 

\subsection{Other studentized estimators and limiting laws.}\label{section:other_limits}
The main construction uses a self-normalized mean with a Gaussian limit. Other inference problems may involve different studentized estimators, and their limiting distributions need not be Gaussian. We therefore separate the general argument for asymptotic e-validity from the threshold comparisons that use the Gaussian law. 

Let $T_n$ be a statistic with a specified null limiting distribution $Y$, and fix $\lambda\in\R$ and $\alpha\geq0$. If, under every distribution in the null model, $T_n\Rightarrow Y$ (convergence in distribution) and $\{(T_n-\lambda)_+^\alpha\}_n$ is uniformly integrable, with $0<\E[(Y-\lambda)_+^\alpha]<\infty$, then
\[
 \frac{(T_n-\lambda)_+^\alpha}{\E[(Y-\lambda)_+^\alpha]}
\]
is an asymptotic e-value, by convergence of expectations. For $\alpha=0$, as usual we interpret $(x-\lambda)_+^0$ as $\mathbf1\{x\geq\lambda\}$ and also require $\Pb(Y=\lambda)=0$. Thus validity extends beyond Gaussian limits, but our explicit normalizers, Gaussian quantile comparisons, and constant $c_\alpha$ are specifically Gaussian. A non-Gaussian law needs its own transformed-survival analysis, and unknown parameters of that law require additional calibration.
For scalar regression coefficients, a studentized asymptotically linear estimator with a standard normal limit and the required uniform integrability admits the same construction. Establishing this condition with estimated nuisance quantities, heteroskedasticity, weak moments, or dependence is a separate probabilistic problem. Multivariate near-optimality and high-dimensional regression are left for future work.

\subsection{Finite-grid mixtures}\label{section:finite_grid}

The experiments in Section~\ref{section:applications} average over finitely many tuning parameters. Such a discrete mixture does not satisfy the positive-density assumption of the continuous-mixture result in Section~\ref{sec:mixtures}. We therefore give a finite-grid counterpart. The grid and its weights are fixed before observing the data, so the mixture remains an asymptotic e-value under Theorem~\ref{theorem:main_theorem}. The question here is how much averaging increases its rejection threshold.

Fix $\alpha\in[0,\infty)$. Recall the components and their one-sided rejection thresholds:
\[
 I_\alpha(\lambda)=\E[(Z-\lambda)_+^\alpha],\qquad
 E^\alpha(z;\lambda)=\frac{(z-\lambda)_+^\alpha}{I_\alpha(\lambda)},\qquad
 U_{u,\alpha}(\lambda)=\inf\{z:E^\alpha(z;\lambda)\geq1/u\},
\]
where $Z\sim N(0,1)$, $u>0$, and $(z-\lambda)_+^0=\mathbf1\{z\geq\lambda\}$. Throughout, $\inf\emptyset=\infty$. Let $G$ denote the survival function of $Z$, and set $c_\alpha=e^\alpha\alpha^{-\alpha}\Gamma(\alpha+1)$ for $\alpha>0$ and $c_0=1$.

For a finite grid $\Lambda_m=\{\lambda_1,\ldots,\lambda_m\}$ and weights $w_j>0$ with $\sum_jw_j=1$, define the mixture and the largest component value by
\[
 E_\pi(z)=\sum_{j=1}^m w_jE^\alpha(z;\lambda_j),\qquad
 M_m(z)=\max_jE^\alpha(z;\lambda_j).
\]
Their rejection thresholds are
\[
 Z_\pi(u)=\inf\{z:E_\pi(z)\geq1/u\},\qquad
 Z_m(u)=\min_jU_{u,\alpha}(\lambda_j).
\]
Thus $Z_m(u)$ is the smallest threshold available on the grid if one could choose the best component separately for each level $u$. For $0<u<1$, also write
\[
 Z_*(u)=\inf_{\lambda\in\R}U_{u,\alpha}(\lambda),\qquad
 D_m(u)=Z_m(u)-Z_*(u)\geq0.
\]
Here $Z_*$ is the unrestricted oracle threshold from Theorem~\ref{proposition:near_optimality}, and $D_m$ is the extra threshold caused by restricting the choice to the grid.

The corollary makes two comparisons. First, a mixture retains at least the smallest weight times the largest component value. Consequently, a grid component that rejects at a sufficiently stricter level also guarantees rejection by the mixture at the target level. The factor $\kappa_m$ below measures the smallest fraction of the largest component retained by the mixture over the relevant rejection region; larger $\kappa_m$ means less loss. Second, to compare with the unrestricted oracle, we add the grid error $D_m$ and apply Theorem~\ref{proposition:near_optimality} to bound the remaining cost of changing the level. These are deterministic threshold comparisons, not finite-sample validity guarantees.

\begin{corollary}[Finite-grid threshold regret]\label{cor:finite_grid}
With the notation above, choose a maximum level $\bar\delta\in(0,1]$ such that
\[
 \mathcal A_{\bar\delta}:=\{z:M_m(z)\geq1/\bar\delta\}\neq\emptyset,
\]
and define
\[
 \kappa_m=\inf_{z\in\mathcal A_{\bar\delta}}\frac{E_\pi(z)}{M_m(z)},
 \qquad w_{\min}=\min_jw_j.
\]
Then $0<w_{\min}\leq\kappa_m\leq1$. For every $0<\delta\leq\bar\delta$, the following hold.
\begin{enumerate}[label=(\roman*),leftmargin=*]
\item \textbf{Comparison with the grid oracle.}
\[
 Z_\pi(\delta)\leq Z_m(\kappa_m\delta)\leq Z_m(w_{\min}\delta).
\]
\item \textbf{Comparison with the unrestricted oracle.} If also $\delta<1/2$ and $Z_m(\kappa_m\delta)<\infty$, then
\[
 Z_\pi(\delta)-Z_*(\delta)
 \leq D_m(\kappa_m\delta)+\frac{\log(c_\alpha/\kappa_m)}{G^{-1}(\delta)}.
\]
\end{enumerate}
For $\alpha=0$, some thresholds may be infinite because the component e-values are bounded. Part~(i) still holds, but its first upper bound is informative only when $Z_m(\kappa_m\delta)$ is finite.
\end{corollary}

\begin{proof}
\textit{The retained fraction.} At any $z$, choose an index $j$ attaining $M_m(z)$. Nonnegativity of the components and the fact that the weights sum to one give
\[
 w_{\min}M_m(z)\leq w_jE^\alpha(z;\lambda_j)
 \leq E_\pi(z)\leq M_m(z).
\]
On $\mathcal A_{\bar\delta}$ the denominator $M_m(z)$ is positive. Dividing and taking the infimum proves $w_{\min}\leq\kappa_m\leq1$. In particular,
\[
 E_\pi(z)\geq\kappa_m M_m(z)\qquad(z\in\mathcal A_{\bar\delta}).
\]

\textit{Part (i).} Put $t=Z_m(\kappa_m\delta)$. If $t=\infty$, the first inequality is immediate. Otherwise, a component attaining the minimum in $Z_m$ rejects at $z=t$. To see why the threshold is attained, note that for $\alpha>0$ each component is continuous and nondecreasing, while for $\alpha=0$ its rejection region is either $[\lambda_j,\infty)$ or empty. Hence
\[
 M_m(t)\geq\frac{1}{\kappa_m\delta}\geq\frac{1}{\bar\delta},
\]
where the last inequality uses $\kappa_m\leq1$ and $\delta\leq\bar\delta$. Thus $t\in\mathcal A_{\bar\delta}$, and
\[
 E_\pi(t)\geq\kappa_m M_m(t)\geq\frac{1}{\delta}.
\]
By the definition of the mixture threshold, $Z_\pi(\delta)\leq t$. For the second inequality, increasing the level can only decrease each component's rejection threshold, and therefore also $Z_m$. Since $\kappa_m\delta\geq w_{\min}\delta$, we obtain $Z_m(\kappa_m\delta)\leq Z_m(w_{\min}\delta)$.

\textit{Part (ii): applying Theorem~\ref{proposition:near_optimality}.} This theorem optimizes over all real tuning parameters. To use it, first separate the error from restricting that optimization to the grid. Part~(i) and the definition of $D_m$ give
\begin{align*}
 Z_\pi(\delta)-Z_*(\delta)
 &\leq Z_m(\kappa_m\delta)-Z_*(\delta)\\
 &=D_m(\kappa_m\delta)+Z_*(\kappa_m\delta)-Z_*(\delta).
\end{align*}
Theorem~\ref{proposition:near_optimality}, applied at levels $\kappa_m\delta$ and $\delta$, respectively, states that
\[
 Z_*(\kappa_m\delta)\leq G^{-1}(\kappa_m\delta/c_\alpha),
 \qquad Z_*(\delta)\geq G^{-1}(\delta).
\]
Substitution bounds the regret by
\[
 D_m(\kappa_m\delta)+G^{-1}(\kappa_m\delta/c_\alpha)-G^{-1}(\delta).
\]

It remains to bound this Gaussian quantile difference. Put $a=\kappa_m\delta/c_\alpha$ and $b=\delta$. Since $\kappa_m\leq1\leq c_\alpha$ and $\delta<1/2$, we have $0<a\leq b<1/2$ and $G^{-1}(u)\geq G^{-1}(b)>0$ for $u\in[a,b]$. If $\phi$ is the standard Gaussian density, then for $t>0$,
$\phi(t)=\int_t^\infty s\phi(s)\,ds\geq tG(t)$. Using this inequality and the derivative of the inverse survival function gives
\begin{align*}
 G^{-1}(a)-G^{-1}(b)
 &=\int_a^b\frac{du}{\phi(G^{-1}(u))}
 \leq\int_a^b\frac{du}{uG^{-1}(u)}\\
 &\leq\frac{\log(b/a)}{G^{-1}(b)}
 =\frac{\log(c_\alpha/\kappa_m)}{G^{-1}(\delta)}.
\end{align*}
Combining the last two bounds proves part~(ii).
\end{proof}

The term $D_m(u)$ measures how much is lost by using the available grid points at level $u$. For $\alpha\geq1$ and $0<u<1$, suppose a global minimizer $\lambda_*(u)$ lies between the smallest and largest grid points. Order the grid increasingly, assume $m\geq2$, and let $h=\max_{j<m}(\lambda_{j+1}-\lambda_j)$ be its largest spacing. A nearest grid point $\widetilde\lambda$ is then at distance at most $h/2$ from $\lambda_*(u)$. The derivative of $U_{u,\alpha}$ vanishes at its minimizer, so Taylor's theorem yields
\begin{align*}
 0\leq D_m(u)
 &\leq U_{u,\alpha}(\widetilde\lambda)-U_{u,\alpha}(\lambda_*(u))\\
 &\leq\tfrac12 C_u|\widetilde\lambda-\lambda_*(u)|^2
 \leq C_u h^2/8,
\end{align*}
where $C_u$ bounds the second derivative of $U_{u,\alpha}$ between these two points. Thus a finer grid reduces this approximation error when it covers the optimizer and the curvature is controlled. For part~(ii), coverage is needed at the \emph{penalized level} $u=\kappa_m\delta$. A guarantee over a range of levels requires coverage and a common curvature bound throughout that range; adding points inside an interval that excludes the optimizer does not remove the error from that exclusion.

For equal weights, the analytic guarantee is $\kappa_m\geq1/m$. For example, a ten-point mixture is guaranteed to reject at level $\delta$ whenever some grid component rejects at level $\delta/10$. A larger $\kappa_m$ gives a less conservative comparison. Table~\ref{table:grid_penalties} evaluates $\kappa_m$ numerically for the actual experimental mixtures: three points with $\bar\delta=0.1$ for post-hoc inference, and ten points with $\bar\delta=0.02$ for multiple testing. The set $\mathcal A_{\bar\delta}$ restricts the calculation to statistic values at which at least one component can reject at the maximum level under consideration. The displayed decimals are numerical estimates, not certified lower bounds; the bound $1/m$ remains available without numerical optimization.

For $\alpha=0$, the ratio $E_\pi(z)/M_m(z)$ is constant between successive grid values, so it suffices to check those constant pieces within $\mathcal A_{\bar\delta}$. For $\alpha=1,2$, grid values mark where components become positive, and crossings between component curves mark where the largest component can change. Between these points the ratio is smooth. The code therefore checks the boundary $Z_m(\bar\delta)$, the grid values and crossings in the rejection region, stationary points of the ratio between them, and its limit as $z\to\infty$. The table also includes the exponential mixtures: their component with the largest tuning parameter dominates as $z\to\infty$, so the ratio tends to $1/m$. Together with the universal weight bound, this proves that their infimum is exactly $1/m$.

\begin{table}[ht]
\centering
\caption{Computed finite-grid mixing penalties on the stated active domains. The decimal values are numerical evaluations, not interval-certified lower bounds.}
\label{table:grid_penalties}
\begin{tabular}{lrr}
\toprule
Family & Post-hoc & Testing \\
\midrule
$\alpha = 0$ & 0.3333 & 0.1000 \\
$\alpha = 1$ & 0.4113 & 0.2817 \\
$\alpha = 2$ & 0.3950 & 0.2855 \\
$\alpha = \infty$ & 0.3333 & 0.1000 \\
\bottomrule
\end{tabular}

\end{table}

\section{Experimental Details} \label{section:experimental_details}

In this section, we provide precise descriptions of the experiments conducted to evaluate the performance of the method of mixtures for both post-hoc inference and multiple testing. All experiments are implemented in Python, utilizing numerical optimization for hyperparameter selection and standard scientific computing libraries for simulations. The code can be found at \href{https://github.com/DMartinezT/bentkus_evalues}{https://github.com/DMartinezT/bentkus\_evalues}.  

\subsection{Post-hoc Inference Experiment} \label{section:posthoc_experiment}

The post-hoc inference experiment, summarized in Table~
\ref{table:posthoc}, compares the sharpness of inference provided by mixtures of Bentkus-type e-values $E_n^\alpha(\theta; \pi_\alpha)$, with $\alpha \in \{0, 1, 2\}$, and exponential-based e-values $E_n^\infty(\theta; \pi_\alpha)$. 


\paragraph{Hyperparameter Selection.} For each e-value type, we consider a mixture $\pi_\alpha$ over three hyperparameters $\lambda_1, \lambda_2, \lambda_3$ that are optimal for target levels $\delta \in \{0.1, 0.05, 0.01\}$. Specifically:
\begin{itemize}
    \item For the Bentkus mixture, we use $\lambda_i = \arg\min_{\lambda} U_{\delta_i, \alpha}(\lambda)$.
    \item For the exponential mixture, we use $\lambda_i = \sqrt{2\log(1/\delta_i)}$.
\end{itemize}
The mixture e-values are defined as $E(Z; \pi) = \frac{1}{3} \sum_{i=1}^3 E(Z; \lambda_i)$.

\paragraph{Procedure.} We evaluate the mixture e-values for observed statistics $Z \in \{2.1, 2.4, 2.7, 3.0\}$. For each $Z$ and each mixture, we calculate the post-hoc level $\delta = 1/E(Z; \pi)$, which represents the tightest significance level at which the null hypothesis can be rejected given the data.

\subsection{Multiple Testing Experiment} \label{section:multiple_testing_experiment}

The multiple testing experiment, presented in Table~
\ref{table:multiple_testing}, reports total rejections for e-BH using mixture e-values.

\paragraph{Experimental Setting.} We simulate a multiple testing problem with $K=100$ hypotheses and a global FDR target $\delta^* = 0.1$. The nominal proportions in Table~\ref{table:multiple_testing} are $p\in\{0.01,0.025,0.05,0.075,0.1\}$. The code uses $K_1=\lfloor pK\rfloor$, giving non-null counts $1,2,5,7,10$. For each simulation,
\begin{itemize}
    \item the null test statistics are sampled from $N(0, 1)$,
    \item the non-null test statistics are sampled from $N(\mu, 1)$ with signal strength $\mu = 3.5$.
\end{itemize}

\paragraph{Mixture Configuration.} We use mixtures over a grid of 10 hyperparameters. The grid is constructed by taking $\lambda_{min}$ and $\lambda_{max}$ to be the optimal hyperparameters for the levels $\delta_{max} = \delta^*\rho = 0.02$ and $\delta_{min} = \delta^*/K = 0.001$, respectively. This range is chosen based on the observation that in most practical e-BH applications, the data-dependent thresholds lie within $[1/\delta^*, K/\delta^*]$. The grid consists of 10 values linearly spaced between the corresponding $\lambda_{min}$ and $\lambda_{max}$.

\paragraph{Evaluation.} Table~\ref{table:multiple_testing} reports average total rejections over 100 simulations per scenario, using seed 42 and the same draws for all methods. Uncertainty estimates are not reported for this Gaussian comparison. The heavy-tailed experiment below reports true and false discoveries, FDR, and Monte Carlo uncertainty over 1000 replications.

\subsection{Heavy-tailed raw-sample experiment}\label{section:heavytail_experiment}

We use $K=100$, five non-nulls, target FDR $\delta^*=0.1$, and independent observations $X_{k,i}=\mu_{k,n}+\varepsilon_{k,i}$ with $\varepsilon_{k,i}\sim t_3/\sqrt3$. This noise has variance one and infinite third absolute moment, and lies in the Gaussian domain of attraction. Set $\mu_{k,n}=3.5/\sqrt n$ for the first five hypotheses and zero for the remaining 95. For $n\in\{50,200,1000\}$, compute the self-normalized statistic actually used in the theory,
\[
 Z_{n,k}(0)=\frac{\sum_{i=1}^n X_{k,i}}{\sqrt{\sum_{i=1}^n X_{k,i}^2}},
\]
and apply the same ten-point mixtures for $\alpha\in\{1,2,\infty\}$. The local alternatives maintain a comparable asymptotic signal across sample sizes. We use 1000 replications, NumPy's default random generator with seed $1234+n$ separately for each sample size, and batches of 50. All families share observations within each replication.

\begin{table}[ht]
\centering
\caption{Raw $t_3/\sqrt3$ observations: mean rejections, true and false discoveries, and empirical FDR over 1000 replications; the target FDR is $0.1$. For rejections and true discoveries, the largest mean is highlighted in {\color{blue}blue} and the smallest in {\color{red}red}. Monte Carlo uncertainty is reported in the accompanying text.}
\label{table:heavytail}
\begin{tabular}{rlrrrr}
\toprule
$n$ & Family & Rejections & True & False & FDR \\
\midrule
50 & $\alpha = 1$ & {\color{blue}3.306} & {\color{blue}3.256} & 0.050 & 0.010 \\
50 & $\alpha = 2$ & 3.231 & 3.184 & 0.047 & 0.010 \\
50 & $\alpha = \infty$ & {\color{red}2.644} & {\color{red}2.633} & 0.011 & 0.003 \\
\midrule
200 & $\alpha = 1$ & {\color{blue}3.360} & {\color{blue}3.295} & 0.065 & 0.014 \\
200 & $\alpha = 2$ & 3.295 & 3.237 & 0.058 & 0.013 \\
200 & $\alpha = \infty$ & {\color{red}2.842} & {\color{red}2.815} & 0.027 & 0.006 \\
\midrule
1000 & $\alpha = 1$ & {\color{blue}3.432} & {\color{blue}3.328} & 0.104 & 0.025 \\
1000 & $\alpha = 2$ & 3.368 & 3.277 & 0.091 & 0.022 \\
1000 & $\alpha = \infty$ & {\color{red}2.846} & {\color{red}2.808} & 0.038 & 0.011 \\
\bottomrule
\end{tabular}

\end{table}

The additional rejections for $\alpha=1,2$ are predominantly true discoveries. Empirical FDR is below $0.1$ in every displayed setting; this finite-$n$ experiment does not establish finite-sample validity. For each replication let $T$ and $V$ denote true and false discoveries, with total $R=T+V$; empirical FDR is the mean of $V/\max(R,1)$. Monte Carlo standard errors (MCSEs) are sample standard deviations (with denominator $999$) divided by $\sqrt{1000}$. Across all rows, MCSEs are at most $0.043$ for rejections, $0.043$ for true discoveries, $0.010$ for false discoveries, and $0.0026$ for FDR. The script \texttt{rebuttal\_heavytail\_experiment.py} writes the LaTeX table and a full-precision CSV containing means and MCSEs for all four metrics.



\end{document}